\documentclass[11pt]{article}
\usepackage[T1]{fontenc}
\makeatletter
\let\@fnsymbol\@arabic
\usepackage{cite} 
\usepackage{graphicx}
\usepackage{amsmath,amssymb}
\usepackage{latexsym, amsthm}
\usepackage{enumerate}
\usepackage{bbm}
\usepackage{hyperref}
\usepackage{authblk}
\usepackage{setspace}
\usepackage{soul}
\usepackage{color}
\large\normalsize

\theoremstyle{plain}
\newtheorem{theorem}{Theorem}[section]
\newtheorem{corollary}[theorem]{Corollary}
\newtheorem{lemma}[theorem]{Lemma}
\newtheorem{proposition}[theorem]{Proposition}
\theoremstyle{definition}
\newtheorem{remark}[theorem]{Remark}

\numberwithin{equation}{section}
\newcommand{\disc}{\mathbf{d}}
\renewcommand{\Re}{\operatorname{Re}} 

\newcommand{\range}{\operatorname{Range}} 
\newcommand{\id}{\operatorname{id}}

\newcommand{\n}{\mathbb{N}}
\newcommand{\<}{\left\langle}  
\renewcommand {\>}{\right\rangle}  
\newcommand{\norma}[1]{\left\|#1\right\|}
\newcommand{\normal}[1]{{\left\vert\kern-0.25ex\left\vert\kern-0.25ex\left\vert #1 
    \right\vert\kern-0.25ex\right\vert\kern-0.25ex\right\vert}}

\newcommand{\pr}{\mathbb{P}}
\newcommand{\B}{\mathcal{B}(X)}

\newcommand{\ew}{\mathbb{E}}
\newcommand{\bpi}{\boldsymbol{\pi}}
\newcommand{\ww}{\mathcal{W}}
\newcommand{\hh}{\mathcal{H}}
\newcommand{\yy}{\overline{Y}}
\newcommand{\pp}{\mathcal{P}}
\newcommand{\dd}{\mathcal{D}}
\newcommand{\bj}[1]{\mathbf{j}_{#1}}
\newcommand{\bs}[1]{\mathbf{s}_{#1}}
\newcommand{\bt}[1]{\boldsymbol{\theta}_{#1}}
\newcommand{\bh}[1]{\mathbf{h}_{#1}}
\newcommand{\bds}[1]{\boldsymbol{\Delta \tau}_{#1}}
\newcommand{\tP}{\widetilde{P}}
\newcommand{\supp}{\operatorname{supp}}
\newcommand{\wlim}{w^*\mbox{-}\lim}

\newcommand{\wcl}{w^*\mbox{-}\operatorname{cl}}
\newcommand{\sbullet}{\raise .5ex\hbox{\tiny$\bullet$}}
\title{\bf Ergodic properties of some piecewise-deterministic Markov
process with application to gene expression modelling }
\author{Dawid Czapla\thanks{E-mail: dawid.czapla@us.edu.pl; tel. +48518582645; Corresponding author.}\;\;
		Katarzyna Horbacz\thanks{E-mail: horbacz@math.us.edu.pl}\;\;
        Hanna Wojew\'odka\thanks{E-mail: hanna.wojewodka@us.edu.pl}}
\affil{\small{ Institute of Mathematics, University of Silesia in Katowice,\\ Bankowa 14, 40-007 Katowice, Poland}}
\date{}

\begin{document}
\maketitle
\vspace*{-0.2cm}
\begin{abstract}
A piecewise-deterministic Markov process, specified by random jumps and switching semiflows, as well as the associated Markov chain given by its post-jump locations, are investigated in this paper. The existence of an exponentially attracting invariant measure and the strong law of large numbers are proven for the chain. Further, a one-to-one correspondence between invariant measures for the chain and invariant measures for the continuous-time process is established. This result, together with the aforementioned ergodic properties of the discrete-time model, is used to derive the strong law of large numbers for the process. The studied random dynamical systems are inspired by certain biological models of gene expression, which are also discussed within this paper.
\end{abstract}
\maketitle
\section*{Introduction}
In this paper we study a subclass of piecewise-deterministic Markov processes (PDMPs), which involve deterministic motion punctuated by random jumps (occuring according to a Poisson process). Due to its wide applications in natural sciences, especially in molecular biology (e.g. models for gene expression \cite{b:HHS, b:mackey_tyran}), the PDMPs have already been widely studied. The research is mainly focused on their long time behaviour and ergodic properties (see e.g. \cite{b:benaim3, b:benaim1, b:costa}).

We are concerned with the PDMP arising from a dynamical system governed by a~specific jump mechanism. Roughly speaking, the deterministic component of the system evolves according to a finite collection of semiflows, which are randomly switched with time. The randomness of post-jump locations stems, however, not only from the semiflows switching  (like in \cite{b:benaim3, b:benaim1}), but also from jumps which occur directly before choosing a new semiflow. Each of these jumps is determined by a randomly selected transfomation of the current state of the system, additionally perturbed by a random shift within an \hbox{$\varepsilon$-neighbourhood}. Such a dynamical system generalises, among others, those developed \hbox{in \cite{b:horbacz_diss}}. It should be stressed that we consider the case where the process evolves over a general phase space, which is not necessarily compact or locally compact (as it is required e.g. in \cite{b:benaim1, b:benaim3, b:costa}). Under these settings, the ergodic properties of the process usually cannot be captured by conventional methods, developed e.g. in \cite{b:meyn}.

The evolution of our dynamical system, which is further denoted by $(\overline{Y}(t))_{t\geq 0}$, can be described in more detail as follows.
The initial state of the system and the index of the semiflow which transforms it are described by arbitrarily distributed random variables $Y_0$ and $\xi_0$, respectively. 
The process is driven by the first flow, i.e.  $Y(t)=S_{\xi_0}(t,Y_0)$, until some random moment $\tau_1$, at which it jumps to a random point in the \hbox{$\varepsilon$-neighbourhood} of $w_{\theta_1}\left(S_{\xi_0}(\Delta \tau_1,Y_0)\right)$, where $(y,\theta)\mapsto w_{\theta}(y)$ is a given continous map, and $\theta_1$ is a random variable depending on $\overline{Y}(\tau_1-)$.
Let $Y_1:=\overline{Y}(\tau_1)=w_{\theta_1}\left(S_{\xi_0}(\Delta \tau_1,Y_0)\right)+H_1$ denote the position of the process directly after this jump. The index of the semiflow that $\overline{Y}(t)$ follows until the next random moment $\tau_2$ is given by $\xi_1$, which depends on both the current state $Y_1$ and the index $\xi_0$ of the previous flow. At the \hbox{time $\tau_2$} the procedure restarts for $(Y_1,\xi_1)$ and is continued inductively. As a result, we obtain a piecewise-deterministic trajectory $\left(\overline{Y}(t)\right)_{t\geq 0}$ with jump times $\tau_1, \tau_2,\ldots$ and post-jump locations $Y_1, Y_2,\ldots$, as illustared in Fig.~\ref{fig}.

\begin{figure}[h!]
\centering
\includegraphics[ scale=0.45]{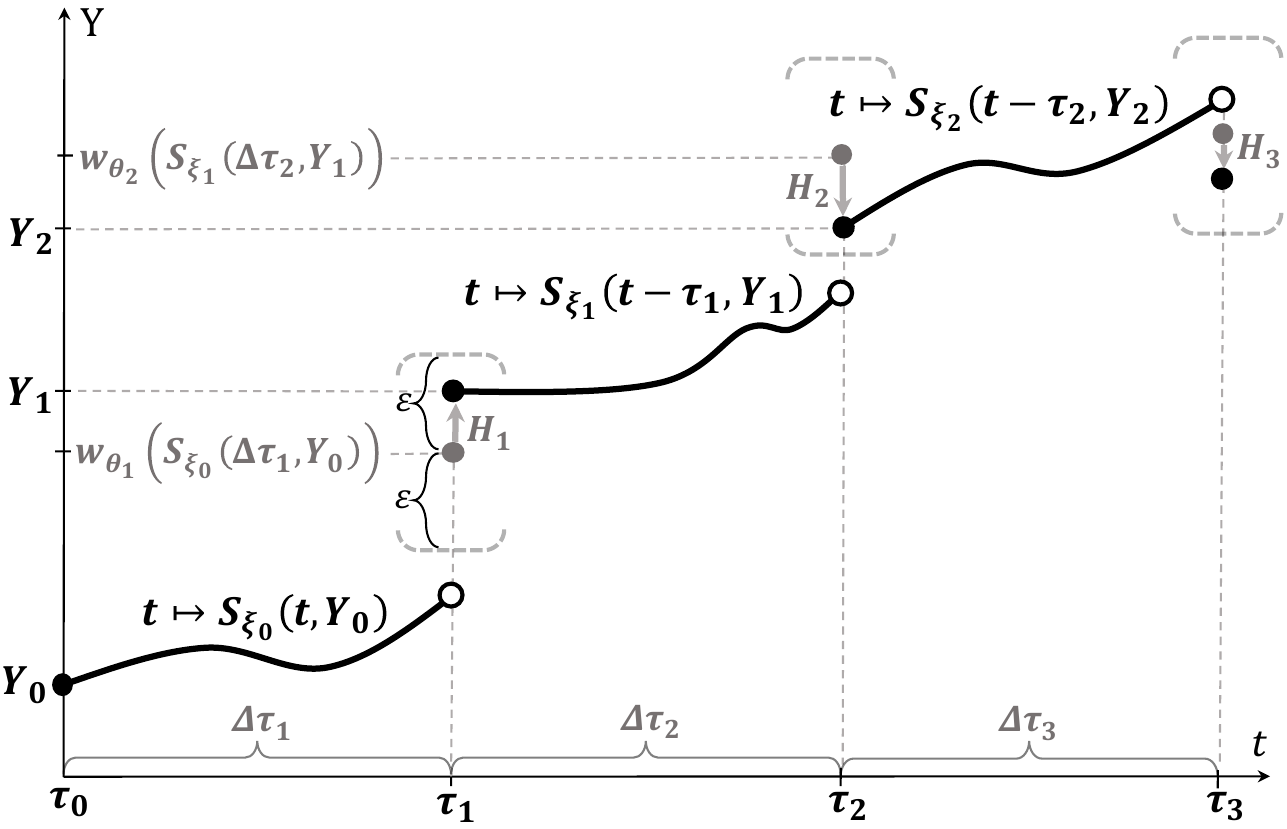}
\caption{
An example trajectory of the process $\left(\overline{Y}(t)\right)_{t\geq 0}$.} 
\label{fig}
\end{figure}

If the collection of semiflows consists of more than one element, then $\left(Y_n\right)_{n\in\n_0}$ and $\left(\overline{Y}(t)\right)_{t\geq 0}$ may not have the Markov property. Therefore, in order to provide the possibility for analysis through the tools of Markov semigroups theory, we investigate the Markov chain $\left(Y_n,\xi_n\right)_{n\in\n_0}$ and the Markov processes  $\left(\overline{Y}(t),\overline{\xi}(t)\right)_{t\geq 0}$, where $\overline{\xi}(t)=\xi_n$ for $t\in\left[\tau_n,\tau_{n+1}\right)$. Clearly, $\left(\overline{Y}(\tau_n),\overline{\xi}(\tau_n)\right)=(Y_n,\xi_n)$ for every $n\in\n_0$.

The main goal of the paper is to provide a set of relatively easily verifiable conditions, which are sufficient to guarantee a certain form of exponential ergodicity of the chain $(Y_n,\xi_n)_{n\in\n_0}$, describing the post-jump locations, as well as the strong law of large numbers (SLLN) for both $(Y_n,\xi_n)_{n\in\n_0}$ and $\left(\overline{Y}(t),\overline{\xi}(t)\right)_{t\geq 0}$. As will be clarified in Section \ref{sec:new}, the conditions imposed on the semiflows, governing the deterministic evolution of the system, are quite naturally met by a wide class of semiflows generated by differential equations (in a reflexive Banach space) involving dissipative operators. An important contribution of our study is also establishing  a one-to-one correspondence between the sets of invariant measures for the process $\left(\overline{Y}(t),\overline{\xi}(t)\right)_{t\geq 0}$ and the chain $(Y_n,\xi_n)_{n\in\n_0}$. 

By the above-mentioned exponential ergodicity of $(Y_n,\xi_n)_{n\in\n_0}$ we mean the existence of a unique invariant distribution, which is exponentially attracting in the dual-bounded Lipschitz distance, also known as the Fortet--Mourier or Dudley metric (see \cite{b:las_frac,b:dudley}). To obtain this, we apply the results of R. Kapica and M. \'Sl\k{e}czka \cite{b:kapica}, which in turn are based on the asymptotic coupling method introduced by M. Hairer~\cite{b:hairer} (applied e.g. in \cite{ b:CLT, b:sleczka, b:woj}). The SLLN for the chain $(Y_n,\xi_n)_{n\in\n_0}$ is shown with the help of the theorem of A. Shirikyan \cite{b:shir}. Having established such properties for the discrete-time model, we further prove that they imply the existence of a unique invariant distribution and the SLLN for the corresponding continuous-time process. Our proofs require the use of several results from the theory of semigroups of linear operators in Banach spaces (see e.g. \cite{b:dynkinb, b:dynkin_mar}), as well as a martingale method (cf. \cite{b:benaim1}). It still remains, however, an open question whether the exponential ergodicity (in the sense described above) of the discrete-time model can imply the analogous property for the associated PDMP.

From the point of view of application, the examined dynamical system provides a useful tool for modelling certain biological processes. For instance, as shown in Section \ref{section:example1}, the process $(Y(t))_{t\geq 0}$ may be adapted as a continuous-time model of prokaryotic gene expression in the presence of transcriptional bursting (cf. \cite{b:mackey_tyran}). It is worth stressing here that, in our framework, the existence of a unique invariant distribution is guaranteed by the aforementioned restrictions imposed on the model. In contrast, applying the results of \cite{b:mackey_tyran}, the invariant measure can only be obtained by solving explicitly some differential equation and proving that its solution is a strictly positive probability density function. The second example, discussed in Section \ref{section:example2}, refers to the discrete-time model for an autoregulated gene, introduced by S.C. Hille et al. \cite{b:HHS}, whose non-disturbed version also appears, for instance, in the cell cycle analysis (cf. \cite{b:cells}). This model constitutes a special case of the system $(Y_n)_{n\in\n_0}$ and indicates the importance of considering a non-locally compact space as the state space in the abstract framework.

The paper is organised as follows. In Section \ref{section2} we introduce basic notation and fundamental concepts on Markov operators (discussed more widely e.g. in \cite{b:las_frac, b:meyn, b:revuz}). Section~\ref{DSdef} provides a detailed description of the model and the principal assumptions employed in the studies. Section \ref{sec:new} is intended to point out a general class of differential equations which generate semiflows consistent with our framework. All the main results are formulated in Section \ref{sec: main_results}, which is divided into two parts: \hbox{Section \ref{sec:main1}}, devoted to the discrete-time model, and Section \ref{sec:cont}, pertaining to its continuous-time interpolation. In Section \ref{section:example} we provide two examples of applications of our abstract framework in the gene expression analysis. Finally, the detailed proofs of all the main results are carried out in Section~\ref{sec:proofs}. Additionally, in the \textnormal{\hyperref[appendix]{Appendix}}, we give a rough sketch of the proof of \hbox{\cite[Theorem 2.1]{b:kapica}}, which serves as an essential tool for the analysis contained in\hbox{ Section~\ref{sec:proof1}}.

\section{Preliminaries}\label{section2}
Let us begin with introducing a piece of notation. Given a metric space $(E,\rho)$, endowed with the Borel $\sigma$-field $\mathcal{B}(E)$, we define

\begin{itemize}
\item[] $B_b(E)$ = the space of all bounded, Borel, real valued functions defined on $E$, endowed with the supremum norm: $\norma{f}_{\infty}=\sup_{x\in E} |f(x)|$, $f\in B_b(E)$;
\item[] $C_b(E)$ = the subspace of $B_b(E)$ consisting of continuous functions;
\item[] $Lip_b(E)$ = the subspace of $B_b(E)$ consisting of Lipschitz continuous functions;
\item[] $B_E(x,r)=\{y\in E:\; \rho(x,y)<r\}$, $r>0$;
\item[] $\mathcal{M}_s(E)$ = the space of all finite, countably additive functions (signed measures) on~$\mathcal{B}(E)$;
\item[] $\mathcal{M}(E)$ = the subset of $\mathcal{M}_s(E)$ consisting of all non-negative measures;
\item[] $\mathcal{M}_1(E)$ = the subset of $\mathcal{M}(E)$ consisting of all probability measures;
\item[] $\mathcal{M}_1^1(E)$ = the set of all $\mu\in\mathcal{M}_1(E)$ satisfying $\int_E \rho(x,x^*)\,\mu(dx)<\infty$ where $x^*$ is an arbitrary (and fixed) point of $E$.
\end{itemize}
Moreover, we use the symbol $\mathbbm{1}_A$ to denote the indicator of $A\subset E$, and define $\mathbb{R}_+:=\left[0,\infty\right)$. 

To simplify notation, in what follows, we write $\<f,\mu\>$ for the integral $\int_E f\,d\mu$, whenever $f:E\to\mathbb{R}$ is a bounded below, Borel measurable function $f:E\to\mathbb{R}$ and $\mu\in\mathcal{M}_s(E)$.

The space $\mathcal{M}(E)$ is assumed to be endowed with the Fortet-Mourier distance \cite{b:las_frac}, defined by:
$$d_{FM}(\mu_1,\mu_2)=\sup\{|\<f,\mu_1-\mu_2\>|:\;f\in \mathcal{R}_{FM}(E)\},\;\;\;\mu_1,\mu_2\in\mathcal{M}(E),$$
where 
\vspace{-0.1cm}
$$
\mathcal{R}_{FM}(E)=\{f\in B_b(E):\; |f|\leq 1,\; |f(x)-f(y)|\leq \rho(x,y)\;\;\mbox{for}\;\;x,y\in E\}.$$
It is well-known (cf. e.g. \cite{b:dudley}) that, whenever $E$ is a Polish space, i.e. a complete separable metric space, then the weak convergence of measures in $\mathcal{M}(E)$ is equivalent to their convergence in the Fortet-Mourier distance \cite{b:las_frac}. We remind here that a sequence $\mu_n\in\mathcal{M}(E)$, $n\in\n$, is \emph{weakly convergent} to $\mu\in\mathcal{M}(E)$ (which is denoted by $\mu_n\stackrel{w}{\to}{\mu}$) whenever $\<f,\mu_n\>\to\<f,\mu\>$ for all $f\in C_b(E)$.

Let us now recall several basic definitions and concepts in the theory of Markov operators, which will be used throughout the paper.

A function $P:E\times\mathcal{B}(E)\rightarrow \left[0,1\right]$ is called a \emph{(sub)stochastic kernel} if for each $A\in\mathcal{B}(E)$, $x\mapsto P(x,A)$ is a measurable map on $E$, and for each $x\in E$, $A\mapsto P(x,A)$ is a~(sub)probability Borel measure on $\mathcal{B}(E)$. 
For an arbitrary (sub)\,stochastic kernel $P$ we consider two operators:
\begin{equation} \label{regp} \mu P(A)=\int_{E} P(x,A)\,\mu(dx)\;\;\;\mbox{for}\;\;\; \mu\in\mathcal{M}(E),\;A\in \mathcal{B}(E),\end{equation}
and
\begin{equation}\label{regd} Pf(x)=\int_{E} f(y)\,P(x,dy)\;\;\;\mbox{for}\;\;\; x\in E,\; f\in B_b(E). \end{equation}
If the kernel $P$ is stochastic, then $(\cdot) P:\mathcal{M}(E) \to \mathcal{M}(E)$ given by \eqref{regp} is called a \emph{regular Markov operator}, and $P(\cdot):B_b(E)\to B_b(E)$ defined by \eqref{regd} is said to be its \emph{dual operator} (see \cite{b:las_frac}). It is easy to check that
\begin{equation} \label{duality} \<f,\mu P\>=\<Pf,\mu\>\;\;\;\mbox{for}\;\;\; f\in B(E),\;\mu\in\mathcal{M}_1(E).\end{equation}
Let us note that $P(\cdot)$, given by \eqref{regd}, can be extended in the usual way to the space of all bounded below Borel functions $\overline{B}_b(E)$ in such a way that \eqref{duality} holds for all $f\in \overline{B}_b(E)$. For notational simplicity, we shall use the same symbol for the extension as for the original operator on $B_b(E)$.

A regular Markov operator $P$ is said to be \emph{Feller} if $Pf\in C_b(E)$ for every $f\in C_b(E)$. A measure $\mu^*\in\mathcal{M}(E)$ is called \emph{invariant} for a Markov operator $P$ if \hbox{$\mu^{*} P = \mu^{*}$}. 
Moreover, we shall say that a probability measure $\mu^*\in\mathcal{M}_1(E)$ is \emph{attracting} whenever
$d_{FM}(\mu P^n, \mu^{*})\to 0$ for any $\mu\in\mathcal{M}_1^1(E)$. If the rate of this convergence is exponential then $\mu^*$ is said to be \emph{exponentially attracting} .

Suppose we are given a time-homogeneous Markov chain \hbox{$(\Phi_n)_{n\in\n_0}$} with state space $E$, defined on a probability space $(\Omega, \mathcal{A},\pr)$. The transition law of this chain is defined by
\begin{equation} \label{trans_n} P(x,A)=\pr(\Phi_{n+1}\in A|\Phi_n=x)\;\;\;\mbox{for}\;\;\;x\in E,\;A\in\mathcal{B}(E),\;n\in\n_0.\end{equation}
Let $(\cdot)P$ denote the Markov operator corresponding to the kernel \eqref{trans_n}. Assuming that $\mu_n$ stands for the distribution of $\Phi_n$, we see that $\mu_{n+1}=\mu_n P$ for all $n\in\n_0$.

A \emph{regular Markov semigroup} $(P^t)_{t\geq 0}$ is a family of regular Markov operators \linebreak \hbox{$P^t: \mathcal{M}(E) \to \mathcal{M}(E)$}, $t\geq 0$, which form a semigroup (under composition) with the identity transformation $P^0$ as the unity element. The semigroup $(P^t)_{t\geq 0}$ is called \emph{Feller} whenever each $P^t$, $t\geq 0$, is Feller. A measure $\mu^*\in\mathcal{M}(E)$ is said to be \emph{invariant} for the Markov semigroup $(P^t)_{t\geq 0}$ if $\mu^{*} P^t = \mu^{*}$ for all $t\geq 0$.

Let {$(\Psi_t)_{t\geq 0}$} be an $E$-valued time-homogeneous Markov process with continuous time parameter $t\in\mathbb{R}_+$. The transition law of {$(\Psi_t)_{t\geq 0}$} is defined by the collection of stochastic kernels of the form
\begin{equation} \label{trans_t} P^t(x,A)=\pr(\psi_{s+t}\in A| \psi_s=x),\;\;\;\mbox{for}\;\;\;x\in E,\;A\in\mathcal{B}(E),\;s,t\geq 0.\end{equation}
Due to the Chapman-Kolmogorov equation, the family $(P^t)_{t\geq 0}$ of Markov operators corresponding to the kernels given by \eqref{trans_t} is then a regular Markov semigroup. 

We will write $\pr_x$ to denote the probability measure $\pr(\cdot\,|\,\Psi_0=x)$ and $\ew_x$ for the expectation with respect to $\pr_x$.

In our further considerations, we also use the concept of \emph{Lyapunov function}. It is defined as a continuous map $V:E\to\left[0,\infty\right)$ which is bounded on bounded sets and satisfies (whenever $E$ is unbounded) $V(x)\to\infty$ as \hbox{$\rho(x,x_0)\to \infty$} for some~$x_0\in E$.

\section{Structure and assumptions of the model} \label{DSdef}
Let $(H,\norma{\cdot})$ be a separable \hbox{Banach} space, and let $Y$ be closed subset of $H$. Further, assume that we are given a finite set \hbox{$I:=\{1,\ldots,N\}$}, endowed with the discrete metric
\begin{equation}
\label{dm}
\disc(i,j)=\begin{cases}
1 &\mbox{for}\;\;\; i\neq j, \\
0 &\mbox{for}\;\;\; i=j,
 \end{cases}
\end{equation}
and a topological measure space $(\Theta,\mathcal{B}(\Theta),\vartheta)$ with a $\sigma$-finite Borel mesure $\vartheta$. For simplicity, in the rest of the paper, we will write $d\theta$ instead of $\vartheta(d \theta)$. 

Let us now consider a collection of semiflows $S_i:\mathbb{R}_+\times Y \to Y$, $i\in I$, where~\hbox{$\mathbb{R}_+:=\left[0,\infty\right)$}, which are continuous with respect to each variable. The \emph{semiflow property} means, as usual,~that
$$S_i(0,y)=y\;\;\;\mbox{and}\;\;\; S_i(s+t,y)=S_i(s,S_i(t,y)))\;\;\;\mbox{for}\;\;\;y\in Y,\; s,t\geq 0.$$
The maps $S_i$ will be switched according to a matrix of continuous functions (probabilities) $\pi_{ij}:Y\to\left[0,1\right]$, $i,j\in I$, satisfying $\sum_{j\in I}\pi_{ij}(y)=1$ for all $y\in Y$ and $i\in I$.
Further, assume that we are given a family $\{w_{\theta}:\,\theta \in \Theta\}$  of transformations from $Y$ to itself, which will be related to the post-jump locations of our dynamical system. We will require that the map $(y,\theta)\mapsto w_{\theta}(y)$ is continuous, and that there exists $\varepsilon^* \in (0,\infty)$ such that
$$w_{\theta}(y)+h\in Y \;\;\;\mbox{whenever}\;\;\; h\in B_H(0,\varepsilon^*),\;\theta \in \Theta,\;y\in Y.$$
Let $p: Y\times \Theta \to \left[0,\infty\right)$ be a continuous map such that $\int_{\Theta} p(y,\theta)\,d\vartheta(\theta)=1$ for any $y\in Y$.  The place-dependent probability density function $\theta\mapsto p(y,\theta)$ will capture the likelihood of occurrence of  $w_{\theta}$ at any jump time.

Now fix $\varepsilon \in \left(0,\varepsilon^* \right]$, and assume that $\nu^{\varepsilon}\in\mathcal{M}_1(H)$ is an arbitrary measure supported on $B_H(0,\varepsilon)$. On a suitable probability space, say $(\Omega, \mathcal{F}, \mathbb{P})$, we define a sequence of random variables {$(Y_n)_{n\in \n_0}$}, taking values in $(Y,\mathcal{B}(Y))$, in such a way that
\begin{equation}\label{def:Y_n}
Y_{n+1}=w_{\theta_{n+1}}(S_{\xi_{n}}(\Delta \tau_{n+1},Y_{n}))+H_{n+1}\;\;\;\mbox{for}\;\;\;n\in\n_0,
\end{equation}
where
\begin{itemize}
\item[\sbullet] $Y_0:\Omega\to Y$ and $\xi_0:\Omega\to I$ are random variables with arbitrary distributions;
\item[\sbullet] $\tau_n :\Omega\to \left[0,\infty\right)$, $n\in \n_0$, form a strictly increasing sequence of random variables with $\tau_0=0$ and $\tau_n\to\infty$, whose increments $\Delta \tau_{n+1}=\tau_{n+1}-\tau_{n}$ are mutually independent and have common exponential distribution with parameter $\lambda>0$;
\item[\sbullet] $H_n:\Omega\to H$, $n\in\n$, are identically distributed random variables with distribution $\nu^{\varepsilon}$;
\item[\sbullet] $\theta_n :\Omega\to \Theta$ and $\xi_n :\Omega\to I$, $n\in \n$, are random variables defined (inductively) so that
$$
\pr(\theta_{n+1}\in D\;|\;S_{\xi_{n}}(\Delta \tau_{n+1},Y_{n})=y)
=\int_{D} p(y,\theta)\,d\theta,
$$
\begin{equation}\label{def:xi_n}
\pr(\xi_{n+1}=j\;|\;Y_{n+1}=y,\, \xi_{n}=i)=\pi_{ij}(y)
\end{equation}
for all $D\in\mathcal{B}(\Theta)$, $y\in Y$, $i,j\in I$ and $n\in\n_0$. Simultaneously, letting
$$\mathcal{G}_0=\sigma(Y_0,\;\xi_0)\;\;\;\mbox{and}\;\;\;\mathcal{G}_n:=\sigma(\mathcal{G}_0\cup\{H_i,\tau_i,\theta_i,\xi_i:\,1\leq i\leq n\})\;\;\;\mbox{for}\;\;\;n\in\n,$$ 
we require that $\theta_{n+1}$ is conditionally independent of the $\mathcal{G}_n$ given \hbox{$S_{\xi_{n}}(\Delta \tau_{n+1},Y_{n})=y$}, and that $\xi_{n+1}$ is conditionally independent of $\mathcal{G}_n$ given $(Y_{n+1},\xi_n)=(y,i)$.

\end{itemize}
Moreover, we assume that, for any $n\in\n_0$, $\Delta\tau_{n+1}$, $H_{n+1}$, $\theta_{n+1}$ and $\xi_{n+1}$ are (\hbox{mutually}) conditionally independent given $\mathcal{G}_n$, and that $\Delta\tau_{n+1}$ and $H_{n+1}$ are independent of $\mathcal{G}_n$.

In our further analysis we shall extensively use the following assumptions:

\begin{itemize}
\item[(A1)]\phantomsection\label{cnd:A1} There exists $y^*\in Y$ such that
$$\sup_{y\in Y}\int_0^{\infty}e^{-\lambda t}\int_{\Theta}\norma{ w_{\theta}(S_i(t,y^*))-y^*}p(S_i(t,y),\theta)\,d\theta\,dt<\infty\;\;\;\mbox{for}\;\;\;i\in I;$$
\item[(A2)]\phantomsection\label{cnd:A2} 
There exist $\alpha\in(-\infty,\lambda)$, $L>0$ and a function \hbox{$\mathcal{L}:Y\to \mathbb{R}_+$}, which is bounded on every  bounded subset of $Y$, such that, for $t\geq 0$, $y_1,y_2\in Y$ and $i,j\in I$,
$$\norma{S_i(t,y_1)-S_j(t,y_2)}\leq Le^{\alpha t}\norma{y_1-y_2}+t\,\mathcal{L}(y_2)\,\disc(i,j),$$
where $\disc(i,j)$ is given by \eqref{dm};
\item[(A3)]\phantomsection\label{cnd:A3}  There exists $L_w>0$ such that
$$\int_{\Theta} \norma{w_{\theta}(y_1)-w_{\theta}(y_2)}p(y_1,\theta)\,d\theta\leq L_w\norma{y_1-y_2}\;\;\;\mbox{for}\;\;\;y_1,y_2\in Y;$$
\item[(A4)]\phantomsection\label{cnd:A4}  There exist $L_{\pi}>0$ and $L_p>0$  such that
$$\sum_{j\in I} |\pi_{ij}(y_1)-\pi_{ij}(y_2)|\leq L_{\pi} \norma{y_1-y_2}\;\;\;\mbox{for}\;\;\;y_1,y_2\in Y,\;i\in I,$$
\begin{equation}
\label{cond_p}
\int_{\Theta} |p(y_1,\theta)-p(y_2,\theta)|\,d\theta\leq L_{p} \norma{y_1-y_2}\;\;\;\mbox{for}\;\;\;y_1,y_2\in Y;
\end{equation}
\item[(A5)]\phantomsection\label{cnd:A5}  There exist $\delta_{\pi}>0$ and $\delta_{p}>0$ such that
$$\sum_{j\in I} \min\{\pi_{i_1,j}(y_1),\pi_{i_2,j}(y_2)\}\geq \delta_{\pi}\;\;\;\mbox{for}\;\;\; t\geq 0,\;i_1,i_2\in I,\; \; y_1,y_2\in Y,$$
$$\int_{\Theta(y_1,y_2)}\min\{p(y_1,\theta),p(y_2,\theta)\}\,d\theta\geq \delta_p\;\;\;\mbox{for}\;\;\;y_1,y_2\in Y,$$
where 
\begin{equation}
\label{defit}
\Theta(y_1,y_2):=\{\theta\in\Theta:\, \norma{w_{\theta}(y_1)-w_{\theta}(y_2)}\leq L_w \norma{y_1-y_2}\}.
\end{equation}
\end{itemize}

Let us now consider a time-homogeneous Markov chain of the form {$(Y_n,\xi_n)_{n\in\n_0}$}, evolving on the space $X:=Y\times I$. We assume that $X$ is equipped with the metric defined~by
\begin{equation}\label{def:metric} \rho_c((y_1,i),(y_2,j))=\norma{y_1-y_2}+c\
\disc(i,j),\;\;\;(y_1,i),(y_2,j)\in X,\end{equation}
where $\disc(i,j)$ is given by \eqref{dm}, and$c>0$ is a sufficiently large constant (specified in Section \ref{sec:proofs}), which depends on the parameters appearing in conditions  \hyperref[cnd:A1]{(A1)}-\hyperref[cnd:A3]{(A3)}. 

The transition law of the chain  $(Y_n,\xi_n)_{n\in\n_0}$, defined on $X\times\mathcal{B}(X)$, will be denoted by $P_{\varepsilon}$. An easy computation shows that, for any $(y,i)\in X$, $A\in\mathcal{B}(X)$,

\begin{align}
P_{\varepsilon}((y,i),A)&:=\int_{B_H(0,\varepsilon)}\int_0^{\infty} \lambda e^{-\lambda t}\int_{\Theta} \Bigl[\sum_{j\in I} f(w_{\theta}(S_i(t,y))+h,j) \nonumber\\
\label{p_eps}
&\quad\times\pi_{ij}(w_{\theta}(S_i(t,y))+h)\Bigr]p(S_i(t,y),\theta)\, d\theta\nu^{\varepsilon}(dh).
\end{align}

Now define the continous-time process  {$\left(\bar{Y}(t),\bar{\xi}(t)\right)_{t\geq 0}$} via interpolation by setting
\begin{equation} \label{zz} \overline{Y}(t)=S_{\xi_n}(t-\tau_n,Y_n),\;\;\;\overline{\xi}(t)=\xi_n\;\;\;\mbox{for}\;\;\;t\in\left[\tau_n,\tau_{n+1}\right),\;n\in \n_0.\end{equation}
It is easy to check that {$\left(\bar{Y}(t),\bar{\xi}(t)\right)_{t\geq 0}$} is a time-homogeneous Markov process, and that $\left(\bar{Y}(\tau_n),\bar{\xi}(\tau_n)\right)=(X_n,\xi_n)$ for $n\in\n_0$. By $(\bar{P}_{\varepsilon}^t)_{t\geq 0}$ we shall denote the Markov semigroup associated with the process $\left(\bar{Y}(t),\bar{\xi}(t)\right)_{t\geq 0}$. The dual operator of $(\cdot)\bar{P}_{\varepsilon}^t$ is then given by 
\begin{equation} \label{pt_eps} \bar{P}^t_{\varepsilon}f(y,i)=\ew(f(\overline{Y}(t),\overline{\xi}(t))\,|\, Y_0=y,\, \xi_0=i)\;\;\;\mbox{for}\;\;\;f\in B_b(X),\,(y,i)\in X.\end{equation}

\section{Reasonableness of the assumptions}\label{sec:new}
It is essential to stress that condition \hyperref[cnd:A2]{(A2)} is reachable by a quite wide class of semiflows acting on reflexive Banach spaces (in particular, Hilbert spaces). As will be clarified below, such semilows can be generated by certain differential equations involving dissipative operators. Furthermore, in many cases, condition \hyperref[cnd:A1]{(A1)} can be then easily derived from the conjunction of \hyperref[cnd:A2]{(A2)} and \hyperref[cnd:A3]{(A3)}. To justify this claim, we first repeat some relevant definitions and results (without proofs) from \cite{b:kazufami}.

Let us recall that $(H, \norma{\cdot})$ is a separable Banach space, and $Y$ stands for a closed subset of~$H$. By $(H^*,\normal{\cdot})$ we shall denote the dual space of $H$ endowed with the operator norm.  For every $u \in H$, we define
$$\mathcal{H}(u):=\{u^*\in H^*:\, u^*u=\norma{u}^2=\normal{u^*}^2=1\}.$$
It follows from the Hahn--Banach theorem that $\mathcal{H}(u)\neq\emptyset$ for every $u\in H$. In the case where $(H,\<\cdot|\cdot\>)$ is a Hilbert space, $\mathcal{H}(u)$ is a singleton consisting of $u^*:=\<\cdot|u\>$ (due to the Frechet--Riesz representation theorem and the fact that any Hilbert space is self-dual). Given a function $A:Y\to H$, we set 
$\range A:=\{u\in H:\,u=Ay\;\;\mbox{for some}\;\;y\in Y\}.$

An operator $A:Y\to H$ (not necessarily linear) is said to be \emph{dissipative} if, for any $y_1,y_2\in Y$, there exists $u^*\in\mathcal{H}(y_1-y_2)$ such that $$\Re u^*(Ay_1-Ay_2)\leq 0.$$
For $\eta\in\mathbb{R}$, the operator $A$ is called \emph{$\eta$-dissipative} whenever $A-\eta \id_Y$ is dissipative, that is
$$\Re u^*(Ay_1-Ay_2)\leq \eta\norma{y_1-y_2}^2.$$
In particular, we see that $A$ is dissipative if and only if it is $\eta$-dissipative for some $\eta\leq 0$. Clearly, if $(H,\<\cdot|\cdot\>)$ is a Hilbert space, then $A$ is $\eta$-dissipative if and only if
$$\<Ay_1-Ay_2|y_1-y_2\>\leq \eta\norma{y_1-y_2}^2\;\;\;\mbox{for any}\;\;\;y_1,y_2\in Y.$$

We now quote the result of Crandall and Liggett \cite{b:crandall}, which shows that $\eta$-dissipative operators generate semiflows with certain decent properties, leading to \hyperref[cnd:A2]{(A2)}.

\begin{remark}[\hspace{-0.15cm}\mbox{\cite[Proposition 1.9]{b:kazufami}}] 
Suppose that $\eta\in\mathbb{R}$, and that $A:Y\to H$ is an \hbox{$\eta$-dissipative} operator. Then, for any $t\in (0,1/|\eta|)$ ({by convention} $1/0:=\infty$), the operator \hbox{$\id_Y-tA:Y\to H_t$}, where $H_t=\range(\id_Y-tA)$, is invertible and $(\id_Y-tA)^{-1}$ is Lipschitz continuous with constant $(1-t\eta)^{-1}$.
\end{remark}

\begin{theorem}[Crandall--Ligget; \hspace{-0.15cm}\text{\cite[Theorem 5.3, Corollary 5.4]{b:kazufami}}] 
\label{crandall}
Let $\eta\in\mathbb{R}$, and suppose that $A:Y\to H$ is an $\eta$-dissipative operator. Further, assume that there exists $T>0$ such that
\begin{equation} \label{rang_cond} Y\subset \range(\id_Y-t A) 
\;\;\;\mbox{for}\;\;\;t\in(0,T).
\end{equation}
Then there exists a semiflow $S:\mathbb{R}_+\times Y\to Y$, which is continuous with respect to each variable and satisfies the following conditions:
\begin{equation} \label{sform} S(t,y)=\lim_{s\to 0^{+}}(\id_Y-sA)^{-\lfloor t/s \rfloor}(y)\;\;\;\mbox{for}\;\;\;y\in Y,\;t\geq 0;\end{equation}
\begin{equation} 
\label{a2A}
\norma{S(t,y_1)-S(t,y_2)}\leq e^{\eta t}\norma{y_1-y_2}\;\;\;\mbox{for}\;\;\;y_1,y_2\in Y,\; t\geq 0;
\end{equation}
\begin{equation} 
\label{a2B}
\norma{S(t,y)-S(s,y)}\leq 
\begin{cases}
(t-s)e^{\eta t}\norma{Ay}, &\eta>0,\\ 
(t-s)e^{\eta s}\norma{Ay}, &\eta\leq 0.\\ 
\end{cases}
\;\;\;\mbox{for}\;\;\;y\in Y,\;0\leq s\leq t.
\end{equation}
\end{theorem}
\noindent

Given $A:Y\to H$ and $y\in Y$, let us now consider the Cauchy problem of the form:
\begin{equation}
\label{cp}
\begin{cases}
u'(t)=Au(t),\;\;t\geq 0,\\
u(0)=y.
\end{cases}
\end{equation}
The following theorem says that, in a reflexive Banach space (e.g. in a Hilbert space), the semiflow specified by \eqref{sform} determines the unique solution of \eqref{cp}.

\begin{theorem}[\hspace{-0.15cm}\text{\cite[Theorem 5.11]{b:kazufami}}]\label{dys_rozw}
Suppose that $H$ is a reflexive Banach space. Further, assume that $A:Y\to H$ is a dissipative operator satisfying  \eqref{rang_cond} for some $T>0$, and let \hbox{$S:\mathbb{R}_+\times Y \to Y$} be the semiflow generated by $A$ in accordance with \eqref{sform}. Then, for any $y\in Y$, the map $\mathbb{R}_+\ni t \mapsto S(t,y)$ is the unique (strong) solution of~\eqref{cp}. 
\end{theorem}

Summarising the above results we can now formulate a conclusion concerning conditions \hyperref[cnd:A1]{(A1)} and \hyperref[cnd:A2]{(A2)}. Recall that the maps $p$ and $w_{\theta}$ are specified within Section~\ref{DSdef}.

\begin{corollary}\label{dys_flow}
Suppose that $H$ is a reflexive Banach space. Further, let $\eta<0$, and assume that $A_i:Y\to H$, $i\in \{1,\ldots,N\}:=I$, are $\eta$-dissipative operators satisfying  \eqref{rang_cond} for some $T>0$. Then there exist semiflows \hbox{$S_i:\mathbb{R}_+\times Y\to Y$}, $i\in I$, which are continuous with respect to each variable, such that, for any $i\in I$ and any $y\in Y$, the map $\mathbb{R}_+\ni t\mapsto S_i(t,y)$ is the unique solution of \eqref{cp} with $A=A_i$. Moreover, the following statements hold:
\begin{itemize}
\item[(1)]\phantomsection \label{cnd:d1} Suppose that there exists $y^*\in Y$ such that
\begin{equation} \label{wcond} \sup_{y\in Y}\int_0^{\infty}e^{-\lambda t}\int_{\Theta} \norma{w_{\theta}(y^*)-y^*} p(S_i(t,y),\theta)\,d\theta\,dt<\infty\;\;\;\mbox{for}\;\;\;i\in I,\end{equation}
and that at least one of the following conditions is fulfilled:
\begin{itemize}
\item[(i)] $p$ does not depend on $y$, i.e. $p(y,\theta)=\bar{p}(\theta)$ for some continuous probability density function $\bar{p}:\Theta\to\left[0,\infty\right)$, and \hyperref[cnd:A3]{(A3)} holds, that is, there exists $L_w>0$ such that 
$$\int_{\Theta} \norma{w_{\theta}(y_1)-w_{\theta}(y_2)}\bar{p}(\theta)\,d\theta\leq L_w\norma{y_1-y_2}\;\;\;\mbox{for}\;\;\;y_1,y_2\in Y;$$
\item[(ii)] there exists $L_w>0$ such that all $w_{\theta}$, $\theta\in\Theta$, are Lipschitz continuous with the~same~constant $L_w$.
\end{itemize}
Then \hyperref[cnd:A1]{(A1)} holds.
\item[(2)]\phantomsection \label{cnd:d2}  Suppose that either $I=\{1\}$ or $A_1,\ldots, A_N$ are bounded on bounded sets. Then  \hyperref[cnd:A2]{(A2)} is satisfied with $L=1$, $\alpha=\eta$ and $\mathcal{L}$ given by
$$\mathcal{L}(y)=
\begin{cases}
0 &\text{in the case where}\; I=\{1\},\\
2\max_{i\in I} \norma{A_i y} &\text{otherwise}. 
\end{cases}
$$
\end{itemize}
\end{corollary}
\begin{proof}
The existence of appropriate $S_1,\ldots,S_N$ follows from Theorem \ref{dys_rozw}. According to Theorem \ref{crandall}, {every} $S_i$ satisfies conditions \eqref{a2A} and \eqref{a2B}, which yield, in particular, that $S_i$ is continuous with respect to each variable and
\begin{equation} \label{c0}\norma{S_i(t,y)-y}\leq t\norma{A_i y}\;\;\;\mbox{for}\;\;\;y\in Y,\,t\geq 0.\end{equation}
In order to show \hyperref[cnd:d1]{(1)}, it suffices to observe that both conditions (i) and (ii) imply
\begin{align*}
\int_{\Theta} \norma{w_{\theta}(S_i(t,y^*))-y^*}&p(S_i(t,y),\theta)\,d\theta\\
&\leq L_w\norma{S_i(t,y^*)-y^*}+\int_{\Theta}\norma{w_{\theta}(y^*)-y^*}p(S_i(t,y), \theta)\,d\theta\\
&\leq L_wt\norma{A_i y^*}+\int_{\Theta}\norma{w_{\theta}(y^*)-y^*}p(S_i(t,y), \theta)\,d\theta.
\end{align*}
Let us now turn to the proof of \hyperref[cnd:d2]{(2)}. In the case where $I=\{1\}$, condition \hyperref[cnd:A2]{(A2)} is just equivalent to \eqref{a2A}. In the general case we apply \eqref{a2A} together with \eqref{c0}.
\end{proof}

\section{Main results} \label{sec: main_results}
\subsection{The Markov chain given by the post-jump locations} \label{sec:main1}

In this part of the paper, we provide a criterion on the existence of a unique invariant probability measure for the operator $P_{\varepsilon}$, corresponding to the chain $(Y_n,\xi_n)_{n\in\n_0}$, which is exponentially attracting in the Fortet--Mourier distance. Having established this, we further obtain, in a relatively simple way, the SLLN (for the discrete-time model).

\begin{theorem}\label{tw:tempo}
Suppose that conditions \textnormal{\hyperref[cnd:A1]{(A1)}-\hyperref[cnd:A5]{(A5)}} hold, and that 
\begin{equation} \label{ll} LL_w+\frac{\alpha}{\lambda}<1.\end{equation}
Then the Markov operator $P_{\varepsilon}$ generated by \eqref{p_eps} has a unique invariant probability measure~$\mu^*$ such that \hbox{$\mu^*\in\mathcal{M}_1^1(X)$}. Moreover, there exist $x^*\in X$ and constants $C\in\mathbb{R}$, \hbox{$\beta\in \left[0,1\right)$} such that
\begin{equation} \label{spec_gap} d_{FM}(\mu P_{\varepsilon}^n, \mu^*)\leq C\beta^n\left(\int_X \rho_c(x^*,x)\,(\mu+\mu^*)(dx)+1\right) \end{equation}
for all $n\in\n$ any any \hbox{$\mu\in\mathcal{M}^1_1(X)$}.
\end{theorem}
The proof of the foregoing result (given in Section \ref{sec:proof1}) is based on the asymptotic coupling method introduced in \cite{b:hairer}. More precisely, we use \cite[Theorem 2.1]{b:kapica}, which gives sufficient conditions for a general Markov chain (in terms of its Markovian coupling) to be exponentially ergodic in the sense described above.

As a straightforward consequence of Theorem \ref{tw:tempo}, we deduce a result that refers to stability of {$(Y_n)_{n\in\n}$} itself (for the proof, see  Section \ref{sec:proof1}). In what follows, we write $\widetilde{\mu}_n$ for the distribution of $Y_n$.

\begin{corollary} \label{col: powrot}
Suppose that the hypotheses of Theorem \ref{tw:tempo} hold. Further, let \hbox{$\mu^*\in\mathcal{M}_1^1(X)$} be the unique invariant probability measure for $P_{\varepsilon}$, and define
$$\widetilde{\mu}^*(B)=\mu^*(B\times I)\;\;\;\mbox{for}\;\;\;B\in\mathcal{B}(Y).$$
Then
\begin{itemize}
\item[(1)]\phantomsection\label{cnd:1_cor} If $Y_0$ has the distribution $\widetilde{\mu}_0=\widetilde{\mu}^{*}$ and
$\pr(\xi_0=i|Y_0=y)=\pi_i(y)$, $i\in I$, $y\in Y$, where $\pi_i$ is the Radon--Nikodym derivative of $\mu^*(\cdot \times\{i\})$ with respect to $\widetilde{\mu}^{*}$, then $\widetilde{\mu}_n=\widetilde{\mu}^{*}$ for all $n\in\n$.
\item[(2)]\phantomsection\label{cnd:2_cor} There exists $\beta\in \left[0,1\right)$ with the property that for every distribution \hbox{$\widetilde{\mu}_0\in\mathcal{M}_1^1(Y)$} of $Y_0$ we may find a constant $\widetilde{C}(\widetilde{\mu}_0)\in\mathbb{R}$ such that
$$d_{FM}(\widetilde{\mu}_n,\widetilde{\mu}^*)\leq \widetilde{C}(\widetilde{\mu}_0)\,\beta^n\;\;\;\mbox{for}\;\;\;n\in\n.$$
\end{itemize}
\end{corollary}

Theorem \ref{tw:tempo} allows us to show the SLLN for the Markov chain $(f(Y_n,\xi_n))_{n\in\n_0}$, where $f\in Lip_b(X)$. This can be done by appealing to a general result of A. Shirikyan \cite{b:shir} (see Section \ref{sec:proof2}). 

\begin{theorem}[SLLN for the discrete-time model]\label{tw:slln-dyskr}
Suppose that conditions \textnormal{\hyperref[cnd:A1]{(A1)}-\hyperref[cnd:A5]{(A5)}} and \eqref{ll} hold. Then, for every \hbox{$f\in Lip_b(X)$} and each initial state $x\in X$, we have
$$\lim_{n\to\infty} \frac{1}{n}\sum_{k=0}^{n-1} f(Y_k,\xi_k)=\<f,\mu^*\>\;\;\;\pr_x \mbox{- almost surely (a.s.)},$$
where $\mu^*$ is the unique invariant distribution for the Markov operator $P_{\varepsilon}$ (which exists by Theorem \ref{tw:tempo}).
\end{theorem}

\subsection{The continuous-time model} \label{sec:cont}

Throughout this section we assume that  $\Theta$, i.e. the set of indexes of the transformations \hbox{$y\mapsto w_{\theta}(y)$}, is endowed with a finite measure $\vartheta$. The main result here asserts that there is a one-to-one correspondence between invariant measures of the operator $P_{\epsilon}$ and invariant measures of the semigroup $(\bar{P_{\varepsilon}}^t)_{t\geq 0}$, which governs the continuous-time process $(\overline{Y}(t),\overline{\xi}(t))_{t\geq 0}$. This, in turn, yields the existence and uniqueness of an invariant probality measure for the semigroup $(\bar{P_{\varepsilon}}^t)_{t\geq 0}$ and enables us to prove the SLLN for the PDMP under consideration.

The above-mentioned correspondence can be described explicitly with the help of the Markov operators associated with the stochastic kernels $G,W: X\times \mathcal{B}(X)\to\left[0,1\right]$ of the forms:
\begin{gather}
\label{def_G}
G((y,i),A)=\int_0^{\infty} \lambda e^{-\lambda t} \mathbbm{1}_A(S_i(t,y),i)\,dt,\\
\label{def_W}
W ((y,i),A)=\sum_{j\in I} \int_{B_H(0,\varepsilon)} \int_{\Theta} \mathbbm{1}_A(w_{\theta}(y)+h,j)\pi_{ij}(w_{\theta}(y)+h)p(y,{\theta})\,d\theta\,\nu^{\varepsilon}(dh).
\end{gather}
Recall that, according to the convention adopted earlier, Markov operators are denoted by the same symbols as those used for the stochastic kernels which generate them.

\begin{theorem}\label{tw:rel}
Let $P_{\varepsilon}$ and $(\bar{P}_{\varepsilon}^t)_{t\geq 0}$ denote the Markov operator and the Markov semigroup corresponding to \eqref{p_eps} and \eqref{pt_eps}, respectively.
\begin{itemize}
\item[(1)]\phantomsection\label{cnd:1r} If $\mu^*\in\mathcal{M}_1(X)$ is an invariant measure for the Markov operator $P_{\varepsilon}$ then $\nu^*:=\mu^* G$ is an invariant  measure for the Markov semigroup $(\bar{P}_{\varepsilon}^t)_{t\geq 0}$ and $\nu^* W = \mu^*$.
\item[(2)]\phantomsection\label{cnd:2r} If $\nu^*\in\mathcal{M}_1(X)$ is an invariant measure for the Markov semigroup $(\bar{P}_{\varepsilon}^t)_{t\geq 0}$ then \hbox{$\mu^*:=\nu^* W$} is an invariant  measure for the Markov operator $P_{\varepsilon}$ and $\mu^* G=\nu^*$.
 \end{itemize}
\end{theorem}
\noindent Our proof of the above theorem, given in Section \ref{sec:proof3}, uses similar techniques to those developed in \cite[Theorem 5.3.1]{b:horbacz_diss} and \cite[Proposition 2.1 and 2.4]{b:benaim1}. 

Combining Theorems \ref{tw:tempo} and \ref{tw:rel} immediately gives
\begin{corollary}\label{inv_exist_pt}
Suppose that conditions \textnormal{\hyperref[cnd:A1]{(A1)}-\hyperref[cnd:A5]{(A5)}} and \eqref{ll} hold. Then $(\bar{P}_{\varepsilon}^t)_{t\geq 0}$, determined {by \eqref{pt_eps}}, has a unique invariant probability measure.
\end{corollary}

Letting $\overline{\mu}_t$ denote the distribution of $\overline{Y}(t)$, $t\geq 0$, we can also easily conclude the analogue of assertion \textnormal{\hyperref[cnd:1_cor]{(1)}} of Corollary~\ref{col: powrot}:
\begin{corollary} \label{col: powrot2}
Suppose that the hypotheses of Corollary \ref{inv_exist_pt} hold, and let $\nu^*\in\mathcal{M}_1(X)$ be the unique invariant probability measure for the Markov semigroup $(\bar{P}_{\varepsilon}^t)_{t\geq 0}$. Define
$$\overline{\nu}^*(B)=\nu^*(B\times I)\;\;\;\mbox{for}\;\;\;B\in\mathcal{B}(Y).$$
If $Y_0$ has the distribution $\overline{\mu}_0=\overline{\nu}^{*}$ and $\pr(\xi_0=i|Y_0=y)=\pi_i(y)$, $i\in I$, $y\in Y$, where $\pi_i$ is the Radon--Nikodym derivative of $\nu^*(\cdot \times\{i\})$ with respect to $\overline{\nu}^{*}$, then $\overline{\mu}_t=\overline{\nu}^{*}$ for all $t\geq 0$.
\end{corollary}

Theorems \ref{tw:slln-dyskr} and \ref{tw:rel} will allow us to prove (see Section \ref{sec:proof4}) a version of the SLLN for the PDMPs $\left(f\left(\overline{Y}(t),\overline{\xi}(t)\right)\right)_{t\geq 0}$ with $f\in Lip_b(X)$.

\begin{theorem}[SLLN for the PDMP] \label{SLLN_cont}
Suppose that conditions \textnormal{\hyperref[cnd:A1]{(A1)}-\hyperref[cnd:A5]{(A5)}} hold with a bounded (or, which is the same thing, constant) \hbox{$\mathcal{L}:Y\to\mathbb{R}_+$}, and that \eqref{ll} is satisfied. Then, for any \hbox{$f\in Lip_b(X)$} and any initial state \hbox{${(y,i)}\in X$}, we have
$$\lim_{t\to\infty} \frac{1}{t} \int_0^t f\left({\overline{Y}(s),\overline{\xi}(s)}\right)\,ds=\<f,\nu^*\>\;\;\;\pr_{(y,i)} - a.s.$$
where $\nu^*$ stands for the unique invariant distribution of  {$\left(\overline{Y}(t),\overline{\xi}(t)\right)_{t\geq 0}$} (which exists by Corollary \ref{inv_exist_pt}).
\end{theorem}

The additional assumption regarding the function $\mathcal{L}$, which appears in \hyperref[cnd:A2]{(A2)}, ensures that the operator $B(X)\ni f\to Gf$ preserves the Lipschitz continuity. This is necessary for our proof method to work, as it enables us to apply Theorem \ref{tw:slln-dyskr} for the Markov chain $\left(Gf({Y_n,\xi_n})\right)_{n\in\n_0}$. Obviously, the above-mentioned requirement is always fulfilled if the process evolves according to only one semiflow.

\section{Applications in gene expression analysis}
\label{section:example}
\subsection{A continuous-time model of prokaryotic gene expression}
\label{section:example1}
Let us describe the dynamical system which occurs in a simple model of gene expression in the presence of transcriptional bursting (cf. \cite{b:mackey_tyran}; for biological aspects, see \cite[Ch.8]{b:biotech1} or \cite[Ch.3]{b:biotech2}). To be more precise, we focus on the prokaryotic (bacterial) gene expression. Genes in prokaryotes are frequently organised in the so-called operons, that is, small groups of related structural genes, which are transcribed at the same time as a unit into a single polycistronic mRNA (which  encodes more than one protein). Typically, the proteins encoded by genes within the same operon interact in some way; for instance, the \emph{lac} operon in bacterium Escherichia coli has three genes involved in the uptake and breakdown of lactose. 

Consider a prokaryotic cell and a single operon containing $d$ structural genes.  Let $t\geq 0$ denote the age of the cell, and suppose that
$\yy(t)=(y_1(t),\ldots,y_d(t))\in\mathbb{R}_+^d$
describes the concentration of $d$ different protein types encoded by the genes within the~operon.

The protein molecues undergo degradation, which is interrupted by transcription occuring in the so-called \emph{bursts}, followed by variable periods of inactivity. As mentioned earlier, the bursts appear simultaneously (at random moments) for all protein types encoded by the genes in the operon. From the biological point of view it is quite natural to assume that the burst onset times, say $0<\tau_1<\tau_2<\ldots$, are separated by exponentially distributed random time intervals $\Delta\tau_1,\Delta\tau_2,\ldots$ (where $\Delta\tau_n:=\tau_n-\tau_n-1$, $\tau_0:=0$) having the same intensity~$\lambda$. Moreover, we require that $\tau_n\to\infty$ (as $n\to\infty$). Since a prokaryotic mRNA can be efficiently transcribed and translated at the same time (because of the lack of nucleus), $\tau_n$ determine the moments of production at once. 

The rate of protein degradation depends on the current amount of the gene product. Moreover, taking into account that every burst may somehow influence the degradation dynamics, we shall assume that the rate is described by a finite collection of continuous vector fields $\dd_i:\mathbb{R}_+^d\to\mathbb{R}_+^d$, $i\in I=\{1,\ldots,N\}$, which are switched randomly from burst to burst. We require that, for some $\alpha<0$, all the operators $-\dd_i$, $i\in I$, are $\alpha$-dissipative, \hbox{i.e.}
\begin{equation} \label{diss} \<y_1-y_2, \dd_i(y_1)-\dd_i(y_2)\>\geq  -\alpha\norma{y_1-y_2}^2\;\;\;\mbox{for}\;\;\;y_1,y_2\in\mathbb{R}_+^d,\; i\in I,
\end{equation}
and that there exists $T>0$ for which
\begin{equation} \label{diss1} \range(\id_{\mathbb{R}_+^d}+t\dd_i)=\mathbb{R}_+^d \;\;\;\mbox{for all}\;\;\;t\in (0,T),\,i\in I.
\end{equation}
A simple example of an operator satisfying such conditions is the map $\left(y_1,\ldots,y_d\right)\mapsto(a_1 y_1,\ldots,a_n y_n)$ with positive $a_1,\ldots,a_d>0$, which can be interpreted as the different degradation rates for each protein. 

It then follows from Theorem \ref{dys_rozw} that there exist semiflows $S_i:\left[0,\infty\right)\times\mathbb{R}_+^d\to\mathbb{R}_+^d$, $i\in I$, such that, for any $i\in I$ and any $y\in Y$, the map $\mathbb{R}_+\ni t\mapsto S_i(t,y)$ is the unique solution of the following Cauchy problem:
\begin{equation} \label{ODE}
\begin{cases}
\yy'(t)=-\dd_i(\overline{Y}(t)),\\
\yy(0)=y.
\end{cases}
\end{equation}
Then the amount of the gene product between two consecutive bursts, say in the time interval $\left[\tau_{n-1},\tau_n\right)$, is described by $t\mapsto S_{\xi_n}(t-\tau_{n-1},\overline{Y}(\tau_{n-1}))$ whenever $\dd_{\xi_n}$ determines the degradation rate in this period. We assume that  $\xi_0,\xi_1,\ldots$, indicating the degradation rates between successive bursts, are $I$-valued random variables, for which the probabilities
\begin{equation}
\label{biol_pi}
\pi_{ij}(y):=\pr(\xi_{n+1}=j\;|\;\overline{Y}(\tau_{n+1})=y,\; \xi_{n}=i),\;\;\;y\in Y,\,i,j\in I,
\end{equation}
does not depend on $n$ and form a matrix consisting of continuous functions of $y$.

Let $\theta_n$ be a random variable with values in $\Theta:=\left[0,m_{\Theta}\right]^d$ (for some positive $m_{\Theta}\in\mathbb{R}$), which describes the amount of  proteins produced (by the genes in the operon) at time $\tau_n$. Clearly, the number of new translations (and so the amount of newly produced proteins) can be different for each of $d$ genes in the operon. The process $(\yy(t))_{t\geq 0}$ then changes from $\yy(\tau_k-)$ to $\yy(\tau_k)=\yy(\tau_k-)+\theta_k$ for $k\in\n$. We assume that $\theta_n$ depends only on the current amount of the gene product. More precisely, we require that
\begin{equation}
\label{biol_p}
\pr(\theta_n\in E \; | \yy(\tau_n-)=y)=\int_E p(y,\theta)\,d\theta,\;\;\;E\in\mathcal{B}(\Theta),
\end{equation}
where \hbox{$p: Y\times \Theta \to \left[0,\infty\right)$} is a continuous function satisfying $\int_{\Theta} p(y,\theta)\,d\theta=1$ for $y\in \mathbb{R}_+^d$. It is quite natural to expect that the variables  $\tau_n$, $\xi_n$ and $\theta_n$ satisfy the independence conditions detailed in Section \ref{DSdef}.

Suppose that the initial amount of the gene product (from the operon) is described by a random variable $\yy(0)$ with an arbitrary (and fixed) distribution. Then, letting $w_{\theta}(y)=y+\theta$ for $\theta\in\Theta$ and $y\in \mathbb{R}_+^d$, we see that for each $n\in\n$ and given $\yy(\tau_n)$, the process $(\yy(t))_{t\geq 0}$ evolves as
\begin{equation}
\label{ge_mod}
\yy(t)=
\begin{cases}
S_{\xi_n}\left(t-\tau_n,\overline{Y}(\tau_n)\right) &\mbox{for}\;\;t\in\left[\tau_n,\tau_{n+1}\right),\\
w_{\theta_{n+1}}\left(S_{\xi_n}(\Delta \tau_{n+1},\overline{Y}(\tau_n)\right)) &\mbox{for}\;\;t=\tau_{n+1}.
\end{cases}
\end{equation}
Such a dynamical system has the same form as $(\overline{Y}(t))_{t\geq 0}$ defined in Section \ref{DSdef}. 

To apply the results of Section \ref{sec: main_results}, observe that the model described above satisfies conditions \textnormal{\hyperref[cnd:A1]{(A1)}-\hyperref[cnd:A3]{(A3)}} and \eqref{ll}. In fact, since all $w_{\theta}$ are Lipschitz continuous with the same constant, and, due to the definition of $\Theta$, $$\int_{\Theta} \norma{w_{\theta}(0)}\,p(S_i(t,y),\theta)\,d\theta=\int_{\Theta} \norma{\theta}\,p(S_i(t,y),\theta)\,d\theta\leq \sqrt{d}\,m_{\Theta}<\infty\;\;\;\mbox{for all}\;\;\;y\in Y,\,i\in I,$$
it follows from Corollary \ref{dys_flow}\hyperref[cnd:d1]{(1)} that  \hyperref[cnd:A1]{(A1)} is satisfied with $y^*=0$. Further, in view of \hbox{Corollary \ref{dys_flow}\hyperref[cnd:d2]{(2)}} and continuity of $\dd_i$, $i\in I$, condition \hyperref[cnd:A2]{(A2)} holds with $L=1$, the dissipativity constant equal to $\alpha$ and $\mathcal{L}(y)=2\max_{i\in I}\norma{\dd_i(y)}$. Finally, \hyperref[cnd:A3]{(A3)} is trivially fulfilled with $L_w=1$. Clearly, for such $L$, $L_w$ and $\alpha$, we also obtain~\eqref{ll}.

Let us now consider the Markov process $(\overline{Y}(t),\overline{\xi}(t))_{t\geq 0}$ with $\overline{\xi}(t):=\xi_n$ for $t\in\left[\tau_n,\tau_{n+1}\right)$, determined by \eqref{ODE}-\eqref{ge_mod}. Assuming that the phase space $\mathbb{R}_+\times I$ is equipped with the metric $\rho_c$, given by \eqref{def:metric}, wherein $c$ is a sufficiently large constant (determined in Section \ref{sec:proofs}), we can use the results of Section \ref{sec:cont}
to provide the SLLN for such a process.
\begin{proposition}
Suppose that the maps $\mathcal{D}_i$, $i\in I$, satisfy conditions \eqref{diss}-\eqref{diss1}, and, additionally, that they are bounded in the case where $I$ consists of more than one element. Further, assume that \hbox{\textnormal{\hyperref[cnd:A4]{(A4)}} and \textnormal{\hyperref[cnd:A5]{(A5)}}} hold for $\{\pi_{ij}:\,i,j\in I\}$ and $p$, determined by \eqref{biol_pi} and \eqref{biol_p}, respectively. Then $(\overline{Y}(t),\overline{\xi}(t))_{t\geq 0}$ has a unique invariant distribution $\nu^*$ such that $\nu^*\in\mathcal{M}_1(\mathbb{R}_+^d)$, and, for any $f\in Lip_b(\mathbb{R}_+^d\times I)$ and any \hbox{$(y,i)\in \mathbb{R}_+^d\times I$},
$$\lim_{t\to\infty}\frac{1}{t} \int_0^t f(\overline{Y}(s),\overline{\xi}(s))\,ds=\<f,\nu^*\>\;\;\;\pr_{(y,i)} \mbox{- a.s.}$$
Moreover, if $\overline{Y}(0)$ has the distribution $\overline{\nu}^*:=\nu^*(\cdot\times I)$, then there exists a probability \hbox{vector $(\pi_i)_{i\in I}$} consisting of $\pi_i:\mathbb{R}_+^d\to\left[0,1\right]$, $i\in I$, such that, whenever
$$\pr(\xi_0=i\,|\, \overline{Y}(0)=y)=\pi_i(y)\;\;\;\mbox{for all}\;\;\;y\in \mathbb{R}_+^d,\;i\in I,$$
then all the variables $\overline{Y}(t)$, $t\geq 0$, are identically distributed.
\end{proposition}

\subsection{A discrete-time model for an autoregluted gene in bacterium}
\label{section:example2}
It is also worth noting that the discrete-time dynamical system of the form \eqref{def:Y_n} includes, as a special case, the abstract model introduced in \cite{b:HHS}, which provide the mathematical background for modelling  the expression of an autoregulated gene in bacterium.

The protein produced from the gene, say A, affects (indirectly) its own expression. Typically, it needs, however, to be first activated, i.e. undergo certain chemical transformations, like phosphorylation (P) and dimerisation (D). \hbox{We can write them symbolically, as}
\begin{equation}
\label{act_sys}
A+P \rightleftarrows AP\;\;\;\mbox{and}\;\;\; 2AP \rightleftarrows D.
\end{equation}
The molecular species A, AP and D diffuse through the cytoplasm of the bacterial cell and are subject to degradation. Suppose that the cytoplasm is represented by an open bounded set $U\subset\mathbb{R}^3$ (with the closure $\overline{U}$), and let $H:=C_b(\overline{U})^3$  be the Banach space endowed with the norm $\norma{y}:=\norma{u}_{\infty}+\norma{v}_{\infty}+\norma{z}_{\infty},$ $y =(u,v,z)\in H$. Further, define \hbox{$H_+:=C_+(\overline{U})^3$,} where $C_+(\overline{U})$ stands for the closed cone of non-negative functions in $C_b(\overline{U})$. The concentration of the compounds A, AP and D can be represented (at any fixed time) by a map \hbox{$y =(u,v,z)\in H_+$}. The deterministic evolution of $y$ (in time $\theta\in \Theta$) is governed by a system of three \emph{reaction-diffusion} differential equations in $C_b(\overline{U})$ with Neumann's type boundary conditions (cf. \cite{b:HHS}). The system has the form
\begin{equation}
\label{act_diff}
\partial_t y(\theta)=D\nabla^2 y(t)+F\circ y(\theta)-\Lambda y(\theta).
\end{equation}
Here \hbox{$F:\mathcal{D}_F\supset\mathbb{R}^3\to \mathbb{R}^3$} is the reaction term (expressed explicitly in \cite{b:HHS}) modelling the activation system~\eqref{act_sys} of A in absence of degradation, and $D,\Lambda\in \mathbb{R}^3\times \mathbb{R}^3$ are diagonal matrices with positive entries on the diagonal, which represent the diffusion constants and degradation rates, respectively, for each of the compounds. Naturally, $\nabla^2$ denotes the Laplace operator. As shown in \cite[Proposition 5.1]{b:HHS}, for each initial condition $y_0\in H_+$, the initial problem associated with \eqref{act_diff} has a unique \emph{mild solution} (see \cite{b:paz}) $\Theta \ni \theta \mapsto \widetilde{\mathcal{W}}(\theta,y_0)\in H_+$, and the corresponding semiflow $\mathcal{\widetilde{W}}$ is continuous.

Suppose that bursts of the gene product A appear at random times $0<\theta_1'<\theta_2'<\ldots$ within the interval $\Theta=\left[0,m_{\Theta}\right]$ (where $m_{\Theta}>0$ is an appropriately large number), and define $\theta_0':=0,\;\;\theta_n:=\theta_n'-\theta_{n-1}'$ for $n\in\n$. Further, let $h_n^A$ denote the amount of protein~A added to the system at time $\theta_n$. We assume that $h_n^A=h_n+\bar{n}f^A$, where $(h_n)_{n\in\n}$ is a seqence of $C_b(\bar{U})$-valued random variables with the same distribution, supported on a ball $B_{C_b(\bar{U})}(0,\varepsilon)$, whilst $f_A\in C(\overline{U})_+\backslash\{0\}$ and $\bar{n}$ stand for some probability density function and some positive integer, respectively. The state $Y_{n+1}$ of the activation system \eqref{act_sys} just after the burst at time $\theta_{n+1}'$ is then determined by the recursive formula
\begin{equation}
\label{hhs1} 
Y_{n+1}=\mathcal{W}(\theta_{n+1}, Y_n)+H_{n+1},\;\;\;n\in\n_0,
\end{equation}
where $\mathcal{W}:\Theta\times H_+ \to H_+$ is given by $\mathcal{W}(\theta,y):=\widetilde{W}(\theta,y )+\bar{n}(f_A,0,0)$, and $H_n=(h_n,0,0)$.
 Clearly, $(H_n)_{n\in\n}$ is then a sequence of random variables with values in $H$ and a common distribution $\nu^{\varepsilon}$, supported on $B_H(0,\varepsilon)$.

The model takes into account (in contrast to the one described in Section~\ref{section:example1}) that the onset time of the burst depends on the current distribution of the gene product. Namely, it is required that
\begin{equation}
\label{hhs1b} 
\pr(\theta_{n+1}\in D\,|\, Y_n=y )=\int_D p(y ,\theta)\,d\theta,\;\;\;D\in \mathcal{B}(\Theta), \;n\in\n,
\end{equation}
where $p:H\times \Theta\to\left[0,\infty\right)$ is a continuous map satisfying $\int_{\Theta} p(y,\theta)\,d\theta=1$ for any $y\in H$. It~is worth noting here that, from biological point of view, each of the maps $p(\cdot,\theta)$, $\theta\in\Theta$, in fact, depends solely on $y_3$, since only the dimer D affects the transcription.

In addition to the above, we also assume that there exists $\bar{\theta}>0$ such that $p(\cdot,\theta)\equiv 0$ for all $\theta\leq \bar{\theta}$ (which is, incidentally, reasonable from the biological viewpoint). Under this assumption, the process $(Y_n)_{n\in\n_0}$ can be restricted to a closed subset of $H_+$, which is crucial for application of our general results. To be more precise, according to \cite[Lemma 5.4]{b:HHS}, there exists a closed subset $Y$ of $H_+$, for which we can choose an $\varepsilon^*\in(0,\infty)$ such that
\begin{equation}
\label{hhs1a} 
\mathcal{W}(\theta,y)+h\in Y\;\;\;\mbox{for all}\;\;\;y\in Y,\;h\in B_H(0,\varepsilon^*),\;\theta\geq\bar{\theta}.
\end{equation}
We therefore assume that $\varepsilon$ (associated with the distribution $\nu^{\varepsilon}$ of $H_n$) belongs to $\left(0,\varepsilon^*\right]$.

The model $(Y_n)_{n\in\n_0}$ for an autoregulated gene, defined by \eqref{hhs1}-\eqref{hhs1a}, can be viewed as the Markov chain determined by the system \eqref{def:Y_n} with $I=\{1\}$, $S_1(t,y)=y$ (for all $t\geq 0$) and $w_{\theta}(y)=\mathcal{W}(\theta,y)$. Appealing to \eqref{p_eps}, we see that the transition law of this chain is given by
$$
\Pi(y,A)=\int_{B(0,\varepsilon)}\int_{\Theta} \mathbbm{1}_{A-\mathcal{W}(\theta,y)}(h)p(y,\theta)\,d\theta\,\nu^{\varepsilon}(dh)=\int_{\Theta} \nu^{\varepsilon}_{\mathcal{W}(\theta,y)}(A)p(y,\theta)\,\nu^{\varepsilon}(dh),
$$
where $\nu^{\varepsilon}_y(\cdot):=\nu^{\varepsilon}(\cdot-y)$ for $y\in Y$, which shows that the system obtained in this way in fact coincides with the one introduced in \cite{b:HHS}. Within the framework described in \cite{b:HHS}, one can show that $p$ satisfies condition \eqref{cond_p}, and that there exists a Borel measurable function $\Lambda: Y\times\Theta\to \mathbb{R}$ \hbox{such that}
\begin{equation}
\label{hhs2}
\norma{\mathcal{W}(\theta,y_1)-\mathcal{W}(\theta,y_2)}\leq \Lambda(y_1,\theta)\norma{y_1-y_2}\;\;\;\mbox{for all}\;\;\;y_1,y_2\in Y,\; \theta\in \Theta,
\end{equation}
and
\begin{equation}
\label{hhs3}
\int_{\Theta}\Lambda(y,\theta)p(y,\theta)\,d\theta\leq \gamma\;\;\;\mbox{for some}\;\;\; \gamma<1.
\end{equation}
From this, it follows immediately that assumptions \hyperref[cnd:A3]{(A3)} and \hyperref[cnd:A4]{(A4)}, introduced in \hbox{Section \ref{DSdef}}, are fulfilled. Clearly, the latter holds with $L_w=\gamma$. Moreover, we see that such a model trivially satisfies condition \hyperref[cnd:A1]{(A1)} (with any $y^*\in Y$) and \hyperref[cnd:A2]{(A2)} with $L=1$, $\alpha=0$ and $\mathcal{L}\equiv 0$. Since $L_w=\gamma<1$, we also obtain \eqref{ll}. In order to get \hyperref[cnd:A5]{(A5)}, we need to additionally assume that there exists $\delta_p>0$ such that, for any $y_1,y_2\in Y,$
\begin{equation}
\label{hhs4}
\int_{\Theta_{\gamma}(y_1)}\min\{p(y_1,\theta),p(y_2,\theta)\}\,d\theta\geq \delta_p,\;\;\;\mbox{where}\;\;\;\Theta_{\gamma}(y):=\{\theta\in\Theta:\,\Lambda(y,\theta)\leq \gamma\}.
\end{equation}
In summary, we can conclude that, whenever conditions \eqref{cond_p} and \eqref{hhs2}-\eqref{hhs4} are satisfied, the results of Section \ref{sec:main1} apply to  the abstract model $(Y_n)_{n\in\n_0}$ defined by \eqref{hhs1}-\eqref{hhs1b} (through identifying $(Y_n)_{n\in \n_0}$ with $(Y_n,1)_{n\in\n_0}$).

Comparing the system presented in Section \ref{section:example1} with the one described above, we see that the second one provides a much more detailed descripiton of the gene expression at time points just after the bursts (by taking into account the diffusion of the phosphorylated and dimerised form of the gene product), whilst the first one allows for modelling the process in continuous time. Interestingly, one can notice that the variables $\theta_n$ play completely different roles in these two models; namely, here, $\theta_n$ describes the burst times, while, in Section \ref{section:example1} it stands for the amount of the proteins produced in the bursts. It is also worth stressing explicitly that, since the model for an autoregulated gene evolves in some space of functions, we are not able to use techniques which are valid only in locally \hbox{compact spaces}.

\section{Proofs}\label{sec:proofs}
In this section, we provide the proofs of all our main results, gathered in Section \ref{sec: main_results}. Before we proceed to the analysis, let us go back to the definition of $\rho_c$, that is, the metric in $X$. As we have stressed in Section \ref{DSdef}, all the results work under the assumption that the \hbox{constant $c$,} appearing in \eqref{def:metric}, is sufficiently large. The choice of  $c$ depends on the parameters appearing in conditions  \hyperref[cnd:A1]{(A1)}-\hyperref[cnd:A3]{(A3)} as follows:
\begin{equation}\label{nc} c\geq  
\max\left\{\frac{1}{\lambda},\,e^{\sup T}\,, \frac{\lambda}{\lambda-\alpha}\right\}\frac{M_{\mathcal{L}}(\lambda-\alpha)}{\lambda L}+ \frac{2(\lambda- \alpha)}{L},  
\end{equation}
where $T\subset \left[0,\infty\right)$ is a fixed bounded set with positive measure such that
\begin{equation}\label{e:remT} e^{\alpha t}\leq \frac{\lambda}{\lambda-\alpha}\;\;\;\mbox{for all}\;\;\;t \in T, \end{equation}
and $M_{\mathcal{L}}:=\sup\{\mathcal{L}(y):\,y\in B_Y(y^*,R)\}$, where $R:=4b/(1-a)$ with 
\begin{equation}\label{def:b}
\begin{split}
a:&=(\lambda LL_w)(\lambda - \alpha)^{-1},\\
b:&= \lambda\max_{i\in I} \sup_{y\in Y}\int_0^{\infty} e^{-\lambda t} \int_{\Theta}\norma{ w_{\theta}(S_i(t,y^*))-y^*}p(S_i(t,y),\theta)\,d\theta\,dt+\varepsilon.
\end{split}
\end{equation}

\subsection{Proofs of the results from Section \ref{sec:main1}}
The proofs of the results referring to the discrete-time model $(Y_n,\xi_n)_{n\in\n_0}$, defined by \eqref{def:Y_n}, are based on two theorems concerning general Markov chains, which we formulate below.

Firstly, we shall quote \cite[Theorem 2.1]{b:kapica}, which relies heavily one the asymptotic coupling method, introduced in \cite{b:hairer} (cf. also \cite{b:woj,b:sleczka}).

For a given stochastic kernel \hbox{$P:E\times \mathcal{B}(E)\to\left[0,1\right]$}, a time-homogeneous Markov chain with values in $E^2$ (endowed with the product topology) is said to be a \emph{{Markovian} coupling} of $P$ whenever its transition law \hbox{$B:E^2\times \mathcal{B}(E^2)\to\left[0,1\right]$} \hbox{satisfies}
$$
B(x,y,A\times E)= P(x,A),\;\;\;B(x,y,E\times A)= P(y,A)\;\;\;\mbox{for}\;\;\;x,y\in E,\;A\in\mathcal{B}(E).
$$
Note that, if $Q: E^2\times\mathcal{B}(E^2)\to\left[0,1\right]$ is a substochastic kernel satisfying
\begin{equation} 
\label{qeq}
Q(x,y,A\times E)\leq P(x,A),\;\;\; Q(x,y,E\times A)\leq P(y,A),\;\;\;x,y\in E,\;A\in\mathcal{B}(E),
\end{equation}
then we can always construct a {Markovian} coupling of $P$ whose transition law $B$ satisfies $Q\leq B$. Indeed, it suffices to define the family $\{R(x,y,\cdot):\;x,y\in E\}$ of measures on $\mathcal{B}(E^2)$, which on rectangles $A\times B\in\mathcal{B}(E^2)$ are given by
$$ R(x,y,A\times B)=
\frac{(P(x,A)-Q(x,y,A\times E))(P(y,B)-Q(x,y,E\times B))}{1-Q(x,y,E^2)}$$
when $Q(x,y,E^2)<1$, and $R(x,y, A\times B)=0$ otherwise. It is then easy to see that \hbox{$B:=Q+R$} is a stochastic kernel satisfying $Q\leq B$, and that the Markov chain with transition function $B$ is a {Markovian} coupling of~$P$.

\begin{theorem}\label{ks-stab}
Let $(E,\rho)$ be a complete separable metric space. Suppose that we are given a regular Markov operator \hbox{$P:\mathcal{M}(E)\to \mathcal{M}(E)$} with the Feller property, and that there exists a substochastic kernel $Q$ on $E^2\times\mathcal{B}(E^2)$ satisfying \eqref{qeq}. Furthermore, assume that the following conditions hold:
\begin{itemize}
\item[(B1)]\phantomsection\label{cnd:B1}There exist a Lyapunov function $V:E\rightarrow\left[0,\infty\right)$ and constants $a\in(0,1)$ and $b>0$ satisfying
$$PV(x)\leq aV(x)+b\;\;\;\mbox{for}\;\;\;x\in E.$$
\item[(B2)]\phantomsection\label{cnd:B2}For some $F\in\mathcal{B}(E^2)$ and some $R>0$ the following conditions are satisfied:
\begin{itemize}
\item[\sbullet] $\supp Q(x,y,\cdot)\subset F$ for $(x,y)\in F$;
\item[\sbullet] There exists a {Markovian} coupling $(\Phi^1_n,\Phi^2_n)_{n\in\n_0}$ of $P$ with transition function $B$, {\hbox{satisfying}} $Q\leq B$, such that for
\begin{equation} \label{zb_k} K:=\{(x,y)\in F:\, V(x)+V(y)<R\} \end{equation}
and $\kappa:=\inf\{n\in\n:\,(\Phi^1_n,\Phi^2_n)\in K\}$ we can choose constants $\zeta\in (0,1)$ and $\bar{C}>0$ so that
{\begin{equation} \label{fastint} \ew_{(x,y)}(\zeta^{-\kappa})\leq \bar{C}\quad\mbox{whenever}\;\;V(x)+V(y)<\frac{4b}{1-a}.\end{equation}}
\end{itemize}
\item[(B3)]\phantomsection\label{cnd:B3}There exists a constant $q\in (0,1)$ such that
$$\int_{E^2}\rho(u,v)\, Q(x,y,du,dv)\leq q\rho(x,y)\;\;\;\mbox{for}\;\;\;(x,y)\in F.$$
\item[(B4)]\phantomsection\label{cnd:B4}Letting $U(r)=\{(x,y):\, \rho(x,y)\leq r\}$ for $r>0$, we have $$\inf_{(x,y)\in F}Q(x,y,U(q\rho(x,y)))>0.$$
\item[(B5)]\phantomsection\label{cnd:B5}There exist constants $l>0$ and $\nu\in\left(0,1\right]$ such that $$Q(x,y,E^2)\geq 1-l\rho(x,y)^{\nu} \;\;\; \mbox{for} \;\;\; (x,y)\in F.$$
\end{itemize}
Then, the operator $P$ possesses a unique invariant measure $\mu^{*}\in\mathcal{M}_1(E)$ such that \hbox{$\<V,\mu^*\><\infty$}. Moreover, there exist constants $C\in\mathbb{R}$ and \hbox{$\beta\in \left[0,1\right)$} such that
\begin{equation} \label{rate} d_{FM}(\mu P^n, \mu^{*})\leq C\beta^n(\<V,\,\mu+\mu^*\>+1) \end{equation}
for all $n\in\n$ and every {$\mu\in\mathcal{M}_1(E)$} satisfying $\<V,\mu\><\infty$. 
\end{theorem}
To explain a bit more the essential idea underlying the above result, we provide a brief sketch of its proof in the \textnormal{\hyperref[appendix]{Appendix}}. 

Secondly, we need a modified version of \cite[Theorem 2.1]{b:shir}. This result is originally stated for Markov chains evolving on a Hilbert space. However, a simple analysis of its proof shows that it can be easily reformulated to the following version, which remains valid in the case of Polish spaces.

\begin{theorem}\label{tw:slln-sh}
Let $(E,\rho)$ be a complete separable metric space, and let $(\Phi_n)_{n\in \n_0}$ be an \hbox{$E$-valued} time-homogeneous Markov chain with transition function $P$. Further, suppose that the following conditions hold:
\begin{itemize}
\item[(C1)]\phantomsection\label{cnd:C1} $P$ has a unique invariant measure $\mu^*\in \mathcal{M}_1(E)$.
\item[(C2)]\phantomsection\label{cnd:C2}  There exist a continuous function $\varphi:E\to\mathbb{R}_+$ and a sequence $(\gamma_k)_{k\in\n_0}$ of positive numbers satisfying $\sum_{k=0}^{\infty} \gamma_k<\infty$, such that for every $f\in Lip_b(E)$ we have 
$$|P^n f(x) - \<f,\mu^*\>|\leq \gamma_n\varphi(x)(\norma{f}_{\infty}+|f|_{Lip})\;\;\;\mbox{for}\;\;\;x\in E,\;n\in\n,$$
where $|f|_{Lip}$ is the minimal Lipschitz constant of $f$.
\item[(C3)]\phantomsection\label{cnd:C3}  there exists a continuous function $h:E\to\mathbb{R}_+$ such that
$$\ew_x\varphi(\Phi_n)\leq h(x)\;\;\;\mbox{for}\;\;\;n\in\n_0,\;x\in E.$$
\end{itemize}
Then for every $f\in Lip_b(E)$ and each initial state $x\in E$
$$\lim_{n\to\infty} \frac{1}{n}\sum_{k=0}^{n-1} f(\Phi_k)=\<f,\mu^*\>\;\;\;\pr_x - \mbox{a.s}.$$
\end{theorem}

\subsubsection{Proofs of Theorem \ref{tw:tempo} and Corollary \ref{col: powrot}} \label{sec:proof1}
The key idea to prove Theorem \ref{tw:tempo}, providing sufficient conditions for the exponential ergodicity of the discrete-time model, is  to verify the hypotheses of Theorem \ref{ks-stab} for the Markov operator $P_{\varepsilon}$, corresponding to \eqref{p_eps}, and an appropriately defined substochastic \hbox{kernel $Q$}.

In order to shorten some of the expressions, needed in the proof, we put
\begin{align}\label{shorts}
\begin{aligned}
&\mathbf{p\,}(x_1,x_2,t,\theta):=p(S_{i_1}(t,y_1),\theta) \wedge p(S_{i_2}(t,y_2),\theta),\\
&\bpi_j(x_1,x_2,t,\theta,h):=\pi_{i_1,j}(w_{\theta}(S_{i_1}(t,y_1))+h)\wedge \pi_{i_2,j}(w_{\theta}(S_{i_2}(t,y_1))+h),\\
&\mathbf{w}_j(x_1,x_2,t,\theta,h):=((w_{\theta}(S_{i_1}(t,y_1))+h,j),(w_{\theta}(S_{i_2}(t,y_2))+h,j)),
\end{aligned}
\end{align}
\noindent for $x_1=(y_1,i_1)$, $x_2=(y_2,i_2)\in X$, $t\geq 0$, $\theta\in \Theta$, $h\in B_H(0,\varepsilon)$, where $\wedge$ denotes minimum. Furthermore, introducing the notation
$$\Delta_h^t f(y,i):=\int_{\Theta} \Bigl[\,\sum_{j\in I} f(w_{\theta}(S_i(t,y))+h,j)\,\pi_{ij}(w_{\theta}(S_i(t,y))+h)\,\Bigr]p(S_i(t,y),\theta)\, d\theta$$
for any $(y,i)\in X$ and $f\in\overline{\mathcal{B}}(X)$, we can write
$$
P_{\varepsilon}f(y,i)=\int_{B_H(0,\varepsilon)}\int_0^{\infty} \lambda e^{-\lambda t}\, \Delta_h^tf(y,i) \, dt\,\nu^{\varepsilon}(dh).
$$

Let us now define
\begin{gather*}
\Lambda_h^t f(x_1,x_2):=\int_{\Theta} \Bigl[\,\sum_{j\in I} f(\mathbf{w}_j(x_1,x_2,t,\theta,h))\,\bpi_j(x_1,x_2,t,\theta,h)\,\Bigr]\mathbf{p\,}(x_1,x_2,t,\theta)\, d\theta,\\
\Lambda_h^t (x_1,x_2,A):=\Lambda_h^t \mathbbm{1}_A(x_1,x_2)\;\;\;\mbox{for}\;\;\; (x_1,x_2)\in X^2,\;f\in \overline{B}_b(X^2),\;A\in\B,
\end{gather*}
and $Q_{\varepsilon}:X^2\times \mathcal{B}(X^2)\to\left[0,1\right]$ by setting
\begin{equation}\label{def-q}
Q_{\varepsilon}(x_1,x_2,A):=\int_{B_H(0,\varepsilon)}\int_0^{\infty} \lambda e^{-\lambda t}\, \Lambda_h^t((x_1,x_2),A)\, dt\,\nu^{\varepsilon}(dh)
\end{equation}
for all $(x_1,x_2)\in X^2$ and $A\in\mathcal{B}(X^2)$. It is easily seen that $Q_{\varepsilon}$ is a~substochastic kernel, satisfying \eqref{qeq} with $P=P_{\varepsilon}$, and obviously
\begin{align*}
Q_{\varepsilon}f(x_1,x_2)=\int_{B_H(0,\varepsilon)}\int_0^{\infty} \lambda e^{-\lambda t}\, \Lambda_h^tf(x_1,x_2)\, dt\,\nu^{\varepsilon}(dh)
\end{align*}
for any $x_1,x_2\in X$ and $f\in\overline{B}_b(X^2)$.

In the analysis that follows, we assume that $X^2$ is equipped with the following metric 
\begin{equation}\label{def:matricX2}
\overline{\rho}_c ((x_1,x_2),(z_1,z_2))=\rho_c(x_1,z_1)+\rho_c(x_2,z_2),\;\;\;(x_1,x_2),(z_1,z_2)\in X^2.
\end{equation}

\begin{proof}[Proof of Theorem \ref{tw:tempo}]
It suffices to verify the hypotheses of Theorem \ref{ks-stab} for $P=P_{\varepsilon}$ and $Q=Q_{\varepsilon}$ given by \eqref{def-q}. 
First of all, let us note that $P_{\varepsilon}$ is Feller, which follows immediately from the continuity of functions  $\pi_{i,j}$, $y\mapsto S_i(t,y)$, $y\mapsto p(y,\theta)$ and $w_{\theta}$. 
Moreover, take $x^*:=(y^*,i^*)$, where $y^*$ is determined by \hyperref[cnd:A1]{(A1)}, and $i^*$ is an arbitrarily fixed element in $I$. Our further reasoning falls naturally into five parts.

{\bf Step 1. } Our first goal is to show that condition \hyperref[cnd:B1]{(B1)} holds for \hbox{$V:X\to\left[0,\infty\right)$} given~by
\begin{equation}\label{defV} V(y,j):=\norma{y-y^*}\;\;\;\mbox{for}\;\;\; (y,j)\in X,\end{equation}
with constants $a$ and $b$ determined by \eqref{def:b}. Clearly, $V$ is a Lyapunov function, which satisfies $V(x)\leq \rho_c(x^*,x)$ for all $x\in X$. Further, note that $a\in (0,1)$, due to \eqref{ll}. For brevity, let us define
$$b_j(t,y):=\int_{\Theta}\norma{w_{\theta}(S_j(t,y^*))-y^*}p(S_j(t,y),\theta)\,d\theta,\;\; j\in I,\,t\geq 0,\,z\in Y.$$
From \hyperref[cnd:A1]{(A1)} we know that $b<\infty$, which, in particular, implies that $b_j(t,y)<\infty$ for almost all $t\geq 0$ and each $y\in Y$.

Let $(y,i)\in X$. By conditions \hyperref[cnd:A3]{(A3)} and \hyperref[cnd:A2]{(A2)} we see that for $t\geq 0$ and~\hbox{$h\in B_H(0,\varepsilon)$},
\begin{align*}
\Delta_h^t V(y,i&)=\int_{\Theta}  \norma{w_{\theta}(S_i(t,y))+h-y^*}\,\Bigl[\,\sum_{j\in I}\pi_{ij}(w_{\theta}(S_i(t,y))+h)\,\Bigr]p(S_i(t,y),\theta)\, d\theta\\
&\leq \int_{\Theta} \norma{w_{\theta}(S_i(t,y))-w_{\theta}(S_i(t,y^*))}\,p(S_i(t,y),\theta)\, d\theta\\
&\quad + \int_{\Theta} \norma{w_{\theta}(S_i(t,y^*))-y^*}\,p(S_i(t,y),\theta)\, d\theta+\norma{h}\\
&\leq L_w\norma{S_i(t,y)-S_i(t,y^*)}+b_i(t,y)+\varepsilon\leq LL_w e^{\alpha t}\norma{y-y^*}+b_i(t,y)+\varepsilon.
\end{align*}
Hence,
\begin{align}\label{eq:<aV+b}
\begin{aligned}
P_{\varepsilon} V(y,i)&=\int_{B_H(0,\varepsilon)}\int_0^{\infty} \lambda e^{-\lambda t}\, \Delta_h^t V(y,i) \, dt\,\nu^{\varepsilon}(dh)\\
& \leq \lambda  L L_w \int_0^{\infty} e^{(\alpha-\lambda)t}\,dt\norma{y - y^*}+\lambda \int_0^{\infty} e^{-\lambda t} b_i(t,y)\,dt+\varepsilon\\
&\leq \frac{\lambda L L_w }{\lambda-\alpha}\norma{y-y^*}+b=a V(y,i)+b.
\end{aligned}
\end{align}

{\bf Step 2. } Let us define $R:=4b/1(1-a)$ and $F:=F_1\cup F_2,$
where $F_1, F_2\subset X^2$ are given by
$$F_1:=\{((y_1,i_1),(y_2,i_2)):\,i_1=i_2\}, \;\;\;F_2:=\{((y_1,i_1),(y_2,i_2)):\,V(y_1,i_1)+V(y_2,i_2)<R\}.$$
We will show that condition \hyperref[cnd:B2]{(B2)} is satisfied for them.

First of all, observe that $\supp Q_{\varepsilon}(x_1,x_2,\cdot)\subset F$ for every $(x_1,x_2)\in X^2$. To see this, let 
$$(x_1,x_2):=((y_1,i_1),(y_2,i_2))\in X^2\;\;\;\mbox{and}\;\;\;(z_1,z_2):=((u_1,k_1),(u_2,k_2))\in X^2\backslash F.$$ 
Then, in particular, $(z_1,z_2)\notin F_1$, that is $k_1\neq k_2$.
Consequently, we obtain 
$$\overline{\rho}_c(\mathbf{w}_j(x_1,x_2,t,\theta,h),(z_1,z_2))\geq c(\disc(j,k_1)+\disc(j,k_2))\geq c\;\;\;\mbox{for}\;\;\; j\in I,\; t\geq 0,\; h\in B_H(0,\varepsilon),$$
where $\mathbf{w}_j$ and $\overline{\rho}_c$ were introduced in \eqref{shorts} \eqref{def:matricX2}, respectively.
Hence, taking \hbox{$\eta\in (0,c)$}, we see that $\mathbf{w}_j(x_1,x_2,t,\theta,h)\notin B_{X^2}((z_1,z_2),\eta)$ for all $j,t,h$, which yields  $Q_{\varepsilon}(x_1,x_2,B_{X^2}((z_1,z_2),\eta))=0$ and thus $(z_1,z_2)\in X^2\backslash \supp Q_{\varepsilon}(x_1,x_2,\cdot)$.

Let $(X^1_n,X^2_n)_{n\geq 0}$ be an arbitrary {Markovian} {coupling of $P_{\varepsilon}$} with transition function $B$ such that $Q_{\varepsilon}\leq B$. Further, put $\kappa:=\inf\{n\in\n:\,(X^1_n,X^2_n)\in K\}$, where $K\subset X^2$ is given by \eqref{zb_k}, and define $\overline{V}:X^2\to \left[0,\infty\right)$ by
$$\overline{V}(x_1,x_2):=V(x_1)+V(x_2)\;\;\;\mbox{for}\;\;\;(x_1,x_2)\in X^2.$$
Since $\{(x_1,x_2)\in X^2:\; \overline{V}(x_1,x_2)<R\}= F_2\subset F$, we obtain
$$\kappa=\inf\left\{n\in \n_0:\, \overline{V}(X^1_n,X^2_n)<R\right\}.$$
Furthermore,  $\overline{V}$ is a Lyapunov function satisfying
$$B\overline{V}(x_1,x_2)\leq a \overline{V}(x_1,x_2)+2b\;\;\;\mbox{for}\;\;\;(x_1,x_2)\in X^2,$$
which follows from \eqref{eq:<aV+b}. 
From \cite[Lemma 2.2]{b:kapica} it now follows that \eqref{fastint} holds. 

{\bf Step 3. }We shall prove that condition \hyperref[cnd:B3]{(B3)} holds with $q:=a\in(0,1)$. Let $(x_1,x_2):=((y_1,i_1),(y_2,i_2))\in F$. By condition \hyperref[cnd:A3]{(A3)} we obtain that, for $t\geq 0$ and \hbox{$h\in B_H(0,\varepsilon)$},
\begin{align} \label{war:A2b}
\begin{split}
\Lambda_h^t \rho_c (x_1,x_2)
&\leq \int_{\Theta} \norma{w_{\theta}(S_{i_1}(t,y_1))-w_{\theta}(S_{i_2}(t,y_2))} p(S_{i_1}(t,y_1),\theta)\,d\theta\\
&\leq L_w\norma{S_{i_1}(t,y_1)-S_{i_2}(t,y_2)}. 
\end{split}
\end{align}
In the case where $i_1\neq i_2$, that is $(x_1,x_2)\in F_2$, it follows that $y_2\in B_Y(y^*,R)$, and thus $\mathcal{L}(y_2)\leq M_{\mathcal{L}}$.
This, combined with \hyperref[cnd:A2]{(A2)}, gives
\begin{equation} \label{war:A3b}
\norma{S_{i_1}(t,y_1)-S_{i_2}(t,y_2)}\leq L e^{\alpha t}\norma{y_1-y_2}+t M_{\mathcal{L}}\,\disc(i_1,i_2).
\end{equation}
Finally, applying \eqref{war:A2b}, \eqref{war:A3b} and \eqref{nc}, we have
\begin{align*}
Q_{\varepsilon}\rho_c (x_1,x_2)
&\leq \lambda  L_w \int_0^{\infty} e^{-\lambda t} \norma{S_{i_1}(t,y_1)-S_{i_2}(t,y_2)}\,dt\\ 
& \leq \lambda LL_w \left (\int_0^{\infty} e^{-\lambda t}\,dt\norma{y_1-y_2}+ \frac{M_{\mathcal{L}}}{L}\int_0^{\infty} te^{-\lambda t} \,dt \,\disc(i_1,i_2) \right)
\\
&\leq \frac{\lambda L L_w }{\lambda-\alpha}\left(\norma{y_1-y_2}+\frac{(\lambda-\alpha)M_{\mathcal{L}}}{\lambda^2 L}\disc(i_1,i_2)\right)\leq q \rho_c(x_1,x_2).
\end{align*}

{\bf Step 4. }We now proceed to prove condition \hyperref[cnd:B4]{(B4)}. For this purpose, let \hbox{$T\subset\left[0,\infty\right)$} be the bounded set with positive measure such that \eqref{e:remT} holds. Clearly, due to \eqref{e:remT} we obtain
\begin{equation}\label{e:llw} LL_w e^{\alpha t} \leq q\;\;\;\mbox{for}\;\;\;t\in T\;\;\;\mbox{(where}\;\;\;q=\lambda LL_w(\lambda - \alpha)^{-1}\,). \end{equation}
Define ${\delta=\delta_{\pi}\delta_p\int_T \lambda e^{-\lambda t}\,dt}$. Letting $(x_1,x_2):=((y_1,i_1),(y_2,i_2))\in F$ and $$U:=\{(u_1,u_2)\in X^2:\, \rho_c(u_1,u_2)\leq q\rho_c(x_1,x_2)\},$$
we shall establish that $Q_{\varepsilon}(x_1,x_2,U)\geq \delta.$

Recall the definition of $\Theta(\cdot,\cdot)$ introduced in \eqref{defit} and consider the following sets:
{$$\mathcal{R}_1(t):=\Theta(S_{i_1}(t,y_1),S_{i_2}(t,y_2)),$$
$$\mathcal{R}_2(t):=\{\theta\in \Theta:\, \norma{w_{\theta}(S_{i_1}(t,y_1)) - w_{\theta}(S_{i_2}(t,y_2)}\leq q \rho_c(x_1,x_2) \}.$$}
Now, applying \eqref{war:A3b}, \eqref{e:llw}  and \eqref{nc}, we see that, for $t\in T$ and $\theta \in \mathcal{R}_1(t)$,
\begin{align*}
\norma{w_{\theta}(S_{i_1}(t,y_1))-w_{\theta}(S_{i_2}(t,y_2))} &\leq L_w \norma{S_{i_1}(t,y_1)-S_{i_2}(t,y_2)}\\
&\leq L_w L e^{\alpha t}\norma{y_1-y_2}+ L_we^{\sup T}M_{\mathcal{L}}\,\disc(i_1,i_2)\\
&\leq q\norma{y_1-y_2}+cq\disc(i_1,i_2)=q\rho_c(x_1,x_2).
\end{align*}
This obviously implies that $\mathcal{R}_1(t)\subset \mathcal{R}_2(t)$ for any $t\in T$. Furthermore, appealing to the notation
introduced in \eqref{shorts}, for any $t\in T$, $\theta \in \Theta$,  $h\in B_H(0,\varepsilon)$ and $j\in I$ we can write 
$\mathcal{R}_2(t)=\{\theta\in\Theta:\, \mathbf{w}_j(x_1,x_2,t,\theta,h)\in U \},$ whence
$$\mathbbm{1}_U(\mathbf{w}_j(x_1,x_2,t,\theta,h))=\mathbbm{1}_{\mathcal{R}_2(t)}(\theta)\geq \mathbbm{1}_{\mathcal{R}_1(t)}(\theta).$$
From \hyperref[cnd:A5]{(A5)} it then follows that, for $t\in T$ and $h\in B_H(0,\varepsilon)$,
\begin{align*}
\Lambda_h^t(x_1,x_2,U)\geq \delta_{\pi} \int_{\mathcal{R}_1(t)} p(S_{i_1}(t,y_1),\theta)) \wedge p(S_{i_2}(t,y_2),\theta)\, d\theta\geq \delta_{\pi}\delta_p,
\end{align*}
which finally gives
\begin{align*}
Q_{\varepsilon}(x_1,x_2,U)\geq \int_{B_H(0,\varepsilon)}\int_T \lambda e^{-\lambda t}\, \Lambda_h^t(x_1,x_2,U)\, dt\,\nu^{\varepsilon}(dh)\geq \delta_{\pi}\delta_p \int_T\lambda e^{-\lambda t}\,dt=\delta.
\end{align*}

{\bf Step 5. }To complete the proof, it remains to establish condition \hyperref[cnd:B5]{(B5)}. Let $(x_1,x_2):=((y_1,i_1),(y_2,i_2))\in F$, and put $z_1(t):=S_{i_1}(t,y_1)$, \hbox{$z_2(t):=S_{i_2}(t,y_2)$}. 
Applying the inequality
$$(s_1\wedge s_2)(t_1\wedge t_2)\geq s_1t_1-s_1|t_1-t_2|-|s_1-s_2|t_1,\;\;s_i,t_i\in\mathbb{R},\,i=1,2,$$
and setting
$$A_h^t(x_1,x_2)=\int_{\Theta}\Bigl[\sum_{j\in I} \pi_{i_1 j}(w_{\theta}(z_1(t))+h)\Bigr] p(z_1(t),\theta) \,d\theta,$$
$$B_h^t(x_1,x_2)=\int_{\Theta}\Bigl[\sum_{j\in I} \pi_{i_1 j}(w_{\theta}(z_1(t))+h) 
\Bigr]|p(z_1(t),\theta)- p(z_2(t),\theta)| \,d\theta,$$
$$
C_h^t(x_1,x_2)=\int_{\Theta}\Bigl[\sum_{j\in I} |\pi_{i_1 j}(w_{\theta}(z_1(t))+h)-\pi_{i_2 j}(w_{\theta}(z_2(t))+h)|\Bigr]p(z_1(t),\theta) \,d\theta,
$$
we infer {(recalling the notation introduced in \eqref{shorts})} that
\begin{align*}
\Lambda_h^t (x_1,x_2,X^2)&=\sum_{j\in I}\int_{\Theta}
\bpi_j(x_1,x_2,t,\theta,h) \mathbf{p\,}(x_1,x_2,t,\theta)d\theta \\
&\geq A_h^t(x_1,x_2)-B_h^t(x_1,x_2)-C_h^t(x_1,x_2)
\;\;\;\mbox{for}\;\;\;t\geq 0,\, h\in B_H(0,\varepsilon).
\end{align*}
Clearly, $A_h^t(x_1,x_2)=1$. Further, due to \hyperref[cnd:A4]{(A4)}, we have
\begin{align*}
B_h^t(x_1,x_2)=\int_{\Theta}|p(z_1(t),\theta)- p(z_2(t),\theta)| \,d\theta\leq L_p\norma{z_1(t)-z_2(t)}.
\end{align*}
Finally, using \hyperref[cnd:A4]{(A4)} and \hyperref[cnd:A3]{(A3)}, sequentially, we observe that
\begin{align*}
C_h^t(x_1,x_2)&\leq \int_{\Theta}\Bigl[\sum_{j\in I} |\pi_{i_1 j}(w_{\theta}(z_1(t))+h)-\pi_{i_1 j}(w_{\theta}(z_2(t))+h)|\Bigr]p(z_1(t),\theta)\,d\theta \\
&\quad+\Bigl\{\int_{\Theta} \Bigl[\sum_{j\in I} \pi_{i_1 j}(w_{\theta}(z_2(t))+h)\,p(z_1(t),\theta)\Bigr]\,d\theta\\
&\quad+\int_{\Theta} \Bigl[\sum_{j\in I} \pi_{i_2 j}(w_{\theta}(z_2(t))+h)\, p(z_1(t),\theta)\Bigr]\,d\theta\Bigr\}\disc(i_1,i_2)\\
&\leq L_{\pi} \int_{\Theta}\norma{w_{\theta}(z_1(t))-w_{\theta}(z_2(t))}p(z_1(t),\theta)\,d \theta+2\disc(i_1,i_2)\\
&\leq L_wL_{\pi}\norma{z_1(t)-z_2(t)}+2\disc(i_1,i_2).
\end{align*}
Hence
$$\Lambda_h^t (x_1,x_2,X^2)\geq 1-(L_p+L_wL_{\pi})\norma{S_{i_1}(t,y_1)-S_{i_2}(t,y_2)}-2\disc(i_1,i_2).$$
Now, reconsidering \eqref{war:A3b} we conclude that
\begin{align*} 
\Lambda_h^t (x_1,x_2,X^2) &\geq 1-(L_p+L_wL_{\pi})\left[\,Le^{\alpha t}\norma{y_1-y_2}+tM_{\mathcal{L}}\,\disc(i_1,i_2)\,\right]-2\disc(i_1,i_2)\\
&\geq 1-L(L_p+L_wL_{\pi}+1)\Bigl[e^{\alpha t}\norma{y_1-y_2}+\frac{1}{L}(tM_{\mathcal{L}}+2)\disc(i_1,i_2)\Bigr]
\end{align*}
for all $t\geq 0$ and $h\in B_H(0,\varepsilon)$. 
From \eqref{nc} it now follows that
\begin{align*}
&Q_{\varepsilon}(x_1,x_2,X^2)=\int_{B_H(0,\varepsilon)}\int_0^{\infty} \lambda e^{-\lambda t}\, \Lambda_h^t(x_1,x_2,X^2)\, dt\,\nu^{\varepsilon}(dh)\\
&\geq 1-\lambda L(L_p+L_wL_{\pi}+1)\left[\int_0^{\infty} e^{(\alpha-\lambda)t}dt\norma{y_1-y_2}+\frac{1}{L}\Bigl(M_{\mathcal{L}}\int_0^{\infty} te^{\lambda t}dt+2\Bigr) \disc(i_1,i_2) \right]\\
&=1-\frac{\lambda L(L_p+L_wL_{\pi}+1)}{\lambda-\alpha}\,\left[\norma{y_1-y_2}+\frac{\lambda-\alpha}{L} \left(\frac{M_{\mathcal{L}}}{\lambda^2}+2\right)\disc(i_1,i_2)\right]\\
&\geq 1-\frac{\lambda L(L_p+L_wL_{\pi}+1)}{\lambda-\alpha}\rho_c(x_1,x_2).
\end{align*}

Summarizing, we have shown that all the hypotheses of Theorem \ref{ks-stab} hold, and thus the proof is now complete.
\end{proof}

\begin{proof}[Proof of Corollary \ref{col: powrot}]
In what follows, for each $n\in\n_0$, we write $\mu_n$ and $\widetilde{\mu}_n$ for the distributions of $(Y_n,\xi_n)$ and $Y_n$, respectively.

In order to show \hyperref[cnd:1_cor]{(1)}, let $\widetilde{\mu}_0=\widetilde{\mu}^{*}$, and choose $\pi_i: Y\to\left[0,\infty\right)$ so that
$$\mu^*(B\times\{i\})=\int_B \pi_i(y)\,\widetilde{\mu}^{*}(dy),\;\; B\in \mathcal{B}(Y),\,i\in I.$$
Since $\sum_{i\in I} \pi_i=1$ almost everywhere with respect to $\widetilde{\mu}_0$, the conditional distribution
$$\pr(\xi_0=i|Y_0=y):=\pi_i(y)\;\;\;\mbox{for}\;\;\;y\in Y,\,i\in I,$$
is well-defined. It follows that
\begin{align*}
\mu_0(B\times J)&=\pr(Y_0\in B,\, \xi_0\in J)=\sum_{j\in J} \int_{\{Y_0\in B\}} \pr(\xi_0=j\,|\,Y_0)\,d\pr\\
&=\sum_{j\in J}\int_B \pi_j(y)\,\widetilde{\mu}^{*}(dy)=\sum_{j\in J}\mu^*(B\times\{j\})=\mu^*(B\times J)
\end{align*}
for all $B\in\mathcal{B}(Y)$ and $J\subset I$, which gives $\mu_0=\mu^*$. Consequently, we obtain for any $B\in\mathcal{B}(Y)$,
$$\widetilde{\mu}_n(B)=\mu_n(B\times I)=\mu_0 P^n (B\times I)=\mu^* P^n (B\times I)=\mu^*(B\times I)=\widetilde{\mu}^*(B).$$

For the proof of assertion \hyperref[cnd:2_cor]{(2)}, it suffices to observe that 
$$d_{FM}(\widetilde{\mu}_n,\widetilde{\mu}^*)\leq d_{FM}(\mu_n,\mu^*)\;\;\;\mbox{for}\;\;\;n\in\n.$$ It follows from the fact that for every $f\in \mathcal{R}_{FM}(Y)$,
\begin{align*}
\left|\<f,\widetilde{\mu}_n-\widetilde{\mu}^*\>\right|&=\left|\int_{Y} f(y)\,(\widetilde{\mu}_n-\widetilde{\mu}^*)(dy)\right|=\left|\int_{Y\times I} f(y)\,(\mu_n-\mu^*)(dy,di)\right|\leq d_{FM}(\mu_n,\mu^*),
\end{align*}
where the last inequality holds since $\left[(y,i)\mapsto f(y)\right]\in \mathcal{R}_{FM}(Y\times I)$. Then \hyperref[cnd:2_cor]{(2)} is ensured \hbox{by \eqref{spec_gap}}.
\end{proof}

\subsubsection{Proof of Theorem \ref{tw:slln-dyskr}}
To establish the SLLN for the discrete-time model, we use both Theorem \ref{tw:tempo}, which has already been proven, and Theorem \ref{tw:slln-sh} (i.e. a version of the SLLN by Shirikyan). 

\begin{proof}[Proof of Theorem \ref{tw:slln-dyskr}] \label{sec:proof2}
It suffices to verify the conditions of Theorem \ref{tw:slln-sh}. By virtue of Theorem~\ref{tw:tempo} there exists a unique invariant distribution $\mu^*$ for $P$, whence \hyperref[cnd:C1]{(C1)} holds, and we can choose  $\bar{C
}\in\mathbb{R}$ and $\beta\in\left[0,1\right)$ such that 
$$d_{FM}(\delta_{(y,i)} P^n, \mu^*)\leq \bar{C}\beta^n(\norma{y-y^*}+1)\;\;\;\mbox{for any}\;\;\;(y,i)\in X,$$
where $\delta_{(y,i)}$ stands for the Dirack measure at $(y,i)$. Let us now define $\varphi(y,i):=V(y,i)+1$ for any $(y,i)\in X$, where $V$ is given by \eqref{defV}. Then
$$|P_{\varepsilon}^n f(y,i)- \<f,\mu^*\>|\leq C\beta^n\varphi(y,i)(\norma{f}_{\infty}+|f|_{Lip})\;\;\;\mbox{for all}\;\;\;(y,i)\in X\;\;\mbox{and}\;\;f\in Lip_b(X),$$
which ensures \hyperref[cnd:C2]{(C2)}. As we have seen in \eqref{eq:<aV+b}, there exists $a\in (0,1)$ and $b>0$ such that 
$P_{\varepsilon}V(y,i)\leq a V(y,i)+b$ for all $(y,i)\in X.$ This gives
$$P^n_{\varepsilon}V(y,i)\leq a^n V(y,i)+b\sum_{k=0}^{n-1} a^k\leq V(y,i)+\frac{b}{1-a}\;\;\;\mbox{for}\;\;\;(y,i)\in X,\;n\in\n,$$
and thus we can conclude that
$$\ew_{(y,i)} \varphi(X_n)=P^n_{\varepsilon} \varphi(y,i)=P^n_{\varepsilon} V(y,i)+1\leq V(y,i)+1+\frac{b}{1-a}\;\;\;\mbox{for}\;\;\;(y,i)\in X.$$
Hence, \hyperref[cnd:C3]{(C3)} holds with $h(y,i):=V(y,i)+1+b(1-a)^{-1}$, $(y,i)\in X,$ which completes the proof.
\end{proof}

\subsection{Proofs of the results from Section \ref{sec:cont}}  \label{sec:proofs2}
We now turn to proving the statements regarding the PDMP defined by \eqref{zz}, which have been gathered in Section \ref{sec:cont}.

\subsubsection{Proof of Theorem \ref{tw:rel}} \label{sec:proof3}

Our proof of Theorem \ref{tw:rel} is based on techniques similar to those employed in \cite[Theorem 5.3.1]{b:horbacz_diss} and\hbox{\cite[Propositions 2.1 and 2.4]{b:benaim1}.}

Let us first recall that $(\bar{P}_{\varepsilon}^t)_{t\geq 0}$ stands for the transition semigroup of $(\overline{Y}(t),\overline{\xi}(t))_{t\geq 0}$, and that $(\Theta, \mathcal{B}(\Theta), \vartheta)$ is assumed to be a  measure space with a finite \hbox{Borel measure $\vartheta$.}

We will begin the analysis by showing that $(\bar{P}_{\varepsilon}^t)_{t\geq 0}$ enjoys the Feller property. Further, we shall provide an approximation of $t^{-1}\bar{P}_{\varepsilon}^t$, which will allow us to establish a relationship between  the weak infinitesimal operators corresponding to the semigroup $(\bar{P}_{\varepsilon}^t)_{t\geq 0}$ and some other semigroup of linear operators, which can be expressed explicitly in a simple form. This observation will be crucial in proving the theorem.

For brevity, we shall write $\bj{n}$, $\bs{n}$, $\bt{n}$ and $\bh{n}$ for the sequences $(j_1,\ldots,j_n)$, $(s_1,\ldots,s_n)$, $(\theta_1,\ldots,\theta_n)$ and $(h_1,\ldots,h_n)$, respectively. The symbols $d \bt{n}$ and $\nu_n^{\varepsilon}(d\,\bh{n})$ will be used to denote the product measures $\vartheta^{\otimes n}(d\theta_1,\ldots,d\theta_n)$ and  $(\nu^{\varepsilon})^{\otimes n}(dh_1,\ldots,dh_n)$, respectively. Moreover, for any \hbox{$n\in\n$}, $y\in Y$, $i\in I$, $\bj{n}\in I^n$, $\bs{n+1}\in \left[0,\infty\right)^{n+1}$, $\bt{n}\in\Theta^n$ and \hbox{$\bh{n}\in B_H(0,\varepsilon)^n,$} we define
\vspace*{-0.1cm}
\begin{gather*}
\ww_1(y,i,s_1,\,\theta_1,h_1)=w_{\theta_1}(S_i(s_1,y))+h_1,\\
\begin{align*}
\ww_n&(y,i,\bj{n-1},\bs{n},\bt{n},\bh{n})=w_{\theta_n}(S_{j_{n-1}}(s_n,
\ww_{n-1}(y,i,\bj{n-2},\bs{n-1},\bt{n-1},\bh{n-1})))+h_n;
\end{align*}
\end{gather*}
\begin{gather*}
\hh_1(y,i,j_1,s_1,s_2,\theta_1,h_1)=
S_{j_1}(s_2,\ww_1(y,i,s_1,\theta_1,h_1)),\\
\hh_n(y,i,\bj{n},\bs{n+1},\bt{n},\bh{n})=S_{j_n}(s_{n+1}, \ww_n(y,i,\bj{n-1},\bs{n},\bt{n},\bh{n}));
\end{gather*}
\begin{gather*}
\Pi_1(y,i,j_1,s_1,\theta_1,h_1)=
\pi_{ij_1}(\ww_1(y,i,s_1,\theta_1,h_1)),\\
\begin{align*}
\Pi_n(y,i,\bj{n},\bs{n},\bt{n},\bh{n})
&=\Pi_{n-1}(y,i,\bj{n-1},\bs{n-1},\bt{n-1},\bh{n-1})\\
&\quad\times\pi_{j_{n-1}j_n}(\ww_n(y,i,\bj{n-1},\bs{n},\bt{n},\bh{n})).
\end{align*}
\end{gather*}
\begin{gather*}
\pp_1(y,i,s_1,\theta_1)=
p(S_i(s_1,y),{\theta_1}),\\
\begin{align*}
\pp_n(y,i,\bj{n-1},\bs{n},\bt{n},\bh{n-1})&=\pp_{n-1}(y,i,\bj{n-2},\bs{n-1},\bt{n-1},\bh{n-2})\\
&\quad\times p(\hh_{n-1}(y,i,{\bj{n-1},\bs{n},\bt{n-1},\bh{n-1}}),\theta_n).
\end{align*}
\end{gather*}
Furthermore, we let $(N_t)_{t\geq 0}$ denote the renewal counting process with arrival times $\tau_n$, i.e.
$$N_t:=\max\{n\in\n_0: \tau_n\leq t\}, \;\;\;t\geq 0.$$

\begin{lemma}\label{lem:lip}
The semigroup $(\bar{P}_{\varepsilon}^t)_{t\geq 0}$ is Feller.
\end{lemma}

\begin{proof}
Let $t\geq 0$ and $f\in C_b(X)$. According to the definition of $N_t$ we have
$$\{N_t=n\}=\{\omega \in \Omega:\; \tau_n(\omega)\leq t<\tau_{n+1}(\omega)\},\;\;\;n\in\n_0.$$
For any $(y,i)\in X$ and any $n\in\n_0$, we define
\begin{align}\label{def:r}
\begin{split}
R_n^t f(y,i):&=\int_{\{N_t=n\}}f(\overline{Y}(t),\overline{\xi}(t))\,d\pr_{(y,i)}=\int_{\{N_t=n\}} f(S_{\xi_{n}}(t-\tau_n,Y_n),\xi_n)\,d\pr_{(y,i)}.
\end{split}
\end{align}
This together with \eqref{pt_eps} allows us to write
\begin{equation}
\label{lip:e6} \bar{P}_{\varepsilon}^t f(y,i)=\ew_{(y,i)} f(\overline{Y}(t),\overline{\xi}(t))=\sum_{n=0}^{\infty} R_n^t f(y,i)\;\;\;\mbox{for}\;\;\;(y,i)\in X.
\end{equation}
We then see that
\begin{equation}\label{lip:e6a} R_0^t f(y,i)=\pr(\tau_1\geq t)f(S_i(t,y),i)=e^{-\lambda t}f(S_i(t,y),i),
\end{equation}
and
\begin{align*}
&R_1^tf(y,i)=\int_{\{N_t=1\}} f(S_{\xi_1}(t-\tau_1, w_{\theta_1}(S_i(\tau_1,y))+H_1),\xi_1)\,d\pr_{(y,i)}\\
&=\int_{\{N_t=1\}} \sum_{j\in I} \int_{B_H(0,\varepsilon)}\int_{\Theta}
f\left(\hh_1(y,i,j,\tau_1,t-\tau_1,\theta,h)\right) \Pi_1(y,i,j,\tau_1,\theta,h)\\
&\quad\times \pp_1(y,i,\tau_1,\theta)\,d\theta \, \nu^{\varepsilon}(dh)\,d\pr.
\end{align*}
In the general case, by setting
\begin{align}\label{def:psi}
\Psi_n f(y,i,&\bs{n},t):=\sum_{\bj{n}\in I^n} \Bigg[\, \int_{B_H(0,\varepsilon)^n}\int_{\Theta^n} \nonumber
f\left(\hh_n(y,i,\bj{n},(\bs{n},t-s_n),\bt{n},\bh{n}),\,j_n\right)\\
&\times \Pi_{n}(y,i,\bj{n},\bs{n},\bt{n},\bh{n})\pp_{n}(y,i,\bj{n-1},\bs{n},\bt{n},\bh{n-1})\,d\bt{n}\, \nu_n^{\varepsilon}(d\bh{n})\Bigg]
\end{align}
for $(y,i)\in X,$ $\bs{n}\in\mathbb{R}_+^n$, and $\bds{n}:=(\Delta \tau_1,\ldots,\Delta \tau_n)$, we obtain
\begin{align} \label{lip:e6b}
&R_n^tf(y,i)=\int_{\{N_t=n\}} \Psi_n f(y,i,\bds{n},t)\,d\pr\;\;\;\mbox{for}\;\;\;(y,i)\in X,\;n\geq 1.
\end{align}
Clearly, $\mathcal{H}_n$, $\Pi_n$ and $\pp_n$ are continuous with respect to $(y,i)$, since  $S_j(\cdot,s)$, $w_{\theta}$, $p(\cdot,\theta)$ and $\pi_{ij}$ are continuous. Therefore, and by the Lebesgue dominated convergence theorem, all the functions $R_n^t f$, $n\in\n_0$, are continuous. Noting that $|R_n^t f(y,i)|\leq \norma{f}_{\infty} \pr(N_t=n)$ for all $n\in\n_0$ and $(y,i)\in X$, we can again apply the Lebesgue theorem (in the discrete version) to deduce that $\bar{P}_{\varepsilon}^t f$ is continuous.
\end{proof}

Using the decomposition \eqref{lip:e6} of $\bar{P}_{\varepsilon}^t$ appearing in the above proof, we can obtain the announced approximation of $\bar{P}_{\varepsilon}^t/t$ {(cf. \cite{b:horbacz_diss})}.

\begin{lemma} \label{lem:pt_apr}
For every $f\in B_b(X)$ there exists a map $u_f:X\times (0,\infty)\to\mathbb{R}$ such that \linebreak $\lim_{t\to 0} \left(\norma{u_f(\cdot,t)}_{\infty}/t\right)=0$
and
\begin{align*}
&\bar{P}_{\varepsilon}^t f(y,i)=e^{-\lambda t} f(S_i(t,y),i)+\lambda e^{-\lambda t}\int_0^t\Psi_1 f(y,i,s,t)\,ds+u_f(y,i,t)\;\;\;\mbox{for}\;\;\;t>0,
\end{align*}
where $\Psi_1$ is defined by \eqref{def:psi}.
\end{lemma}
\begin{proof}
According to \eqref{lip:e6}, we may write $\bar{P}_{\varepsilon}^t f=R_0^t f+R_1^t f + \sum_{n=2}^{\infty}R_n^t f,$
where $R_n^t f$ are defined by \eqref{def:r}. As we have already seen, $R_0^tf$ can be expressed as \eqref{lip:e6a}. Since the distribution of $\tau_1$ conditional on $\{N_t=1\}$ is uniform over $(0,t)$, it follows that

\begin{align*}
R_1^t f(y,i)&=\int_{\{N_t=1\}} \Psi_1 f(y,i,\bds{1},t)\,d\pr
=\pr(N_t=1)\int_{\Omega} \Psi_1 f(y,i,\bds{1},t)\,d\pr(\cdot\,|\,N_t=1)\\
&=\lambda t e^{-\lambda t}\int_0^t \frac{1}{t}\Psi_1 f(y,i,s,t)\,ds=\lambda e^{-\lambda t}\int_0^t \Psi_1 f(y,i,s,t)\,ds.\\
\end{align*}
Define $u_f(y,i,t):=\sum_{n=2}^{\infty}R_n^t f(y,i)$ for $(y,i)\in X$ and $t> 0$. By \eqref{lip:e6b} we obtain
\begin{align*}
\frac{|u_f(y,i,t)|}{t}&\leq \norma{f}_{\infty}\frac{1}{t}\sum_{n=2}^{\infty}\pr(N_t=n)=\norma{f}_{\infty}\frac{1}{t}e^{-\lambda t}\sum_{n=2}^{\infty} \frac{(\lambda t)^n}{n!}\\
&=\norma{f}_{\infty}\frac{1}{t}e^{-\lambda t}(e^{\lambda t}-1-\lambda t)=\norma{f}_{\infty}\left(\frac{1-e^{-\lambda t}}{t}-\lambda e^{-\lambda t} \right)
\end{align*}
for all $t>0$ and $(y,i)\in X$, whence $\lim_{t\to 0}\, (\norma{u_f(\cdot,t)}_{\infty}/t)=0$, as claimed.
\end{proof}

Having established this, we immediately obtain the following:

\begin{lemma}\label{lem:st_cont}
The semigroup $(\bar{P}_{\varepsilon}^t)_{t\geq 0}$ is stochastically continuous, i.e. 
$$\lim_{t\to 0} \bar{P}_{\varepsilon}^t f(x)=f(x)\;\;\;\mbox{for all} \;\;\; x\in X\;\;\;\mbox{and}\;\;\;f\in C_b(X).$$
\end{lemma}

In order to prepare for the proof of Theorem \ref{tw:rel}, we also need a few facts concerning semigroups of linear operators on Banach spaces, adapted from \cite{b:dynkinb, b:dynkin_mar}.

Consider the Banach space $(\mathcal{M}_s(X), \norma{\cdot}_{TV})$, where
$$\norma{\mu}_{TV}:=\sup\{|\<f,\mu\>|:\,f\in B_b(E),\;\norma{f}_{\infty}\leq 1\},\;\;\;\mbox{for}\;\;\;\mu\in\mathcal{M}_s(E),$$
and let $\mathcal{M}_s(X)^*$ denote its dual space. For any function $f\in B_b(X)$ we define the functional $\ell_f:\mathcal{M}_s(X)\to\mathbb{R}$ by
$$\ell_f(\mu):=\<f,\mu\>\;\;\;\mbox{for}\;\;\;\mu\in\mathcal{M}_s(X).$$
Clearly, $\ell_f\in \mathcal{M}_s(X)^*$, and $f\mapsto \ell_f$ is an isometric embedding of $B_b(X)$ in $\mathcal{M}_s(X)^*$, i.e. an injective linear map satisfying $\norma{\ell_f}=\norma{f}_{\infty}$ for all $f\in B_b(X)$. Therefore, $B_b(X)$ can be regarded as a subspace of $\mathcal{M}_s(X)^*$, and consequently, it can be endowed with the weak star ($w^*$-) topology inherited from $\mathcal{M}_s(X)^*$. Moreover, one can easily show that $B_b(X)$ is $w^*$-closed in $\mathcal{M}_s(X)^*$.


We say that a sequence $f_n\in B_b(X)$, $n\in\n$, converges $*$-weakly to $f\in B_b(X)$ and we write $\wlim_{n\to \infty} f_n=f$ whenever $(\ell_{f_n})_{n\in\n}$ converges $*$-weakly to $\ell_f$ in $\mathcal{M}_s(X)^*$, that is 
\begin{equation} \label{wlimit} \wlim_{n\to \infty} f_n=f \;\;\;\mbox{iff}\;\;\;\lim_{n\to\infty}\<f_n,\mu\>=\<f,\mu\>\;\;\;\mbox{for all}\;\;\;\mu\in\mathcal{M}_s(X).\end{equation}
It is not hard to show that  $\wlim_{n\to \infty} f_n=f$ is equivalent to the requirement that \hbox{$f_n(x)\to f(x)$} for $x\in X$, and the sequence $(\norma{f_n}_{\infty})_{n\in\n}$ is bounded.

Suppose we are given a subspace $L$ of $B_b(X)\subset \mathcal{M}_s(X)^*$ and a contraction semigroup $(T^t)_{t\geq 0}$ of bounded linear \hbox{operators} $T^t:L\to L$, $t\geq 0$.
Let
\begin{equation}\label{defL0} L_0(T):=\{f\in L:\,\wlim_{t\to 0} T^t f =f\}.\end{equation}
We can consider {the} so-called \emph{weak infinitesimal operator} \cite[Ch.1\,\S\,6]{b:dynkinb} of the semigroup $(T^t)_{t\geq 0}$. It is the function $A:D(A)\to L_0(T)$ defined by
$$Af=\wlim_{t\to 0}\frac{T^t f -f}{t}\;\;\;\mbox{for}\;\;\;f\in D(A),$$
where
$$D(A):=\left\{f\in L:\, \wlim_{t\to 0}\frac{T^t f -f}{t}\;\;\;\mbox{exists and belongs to}\;L_0(T)\right\}.$$
Clearly, $D(A)\subset L_0(T)$. We now give several properties of such an operator, which are useful for our further considerations.

\begin{remark}\label{rem:inf}
Under the assumptions made above the following statements hold (see \cite[p. 40]{b:dynkinb} or \cite[pp. 437-448]{b:dynkin_mar} for the proofs):
\begin{itemize}
\item[(i)]\phantomsection\label{cnd:rem1} $\wcl D(A) =\wcl L_0(T),$ where $\wcl$ denotes the $*$-weak closure in $B_b(X)$.
\item[(ii)]\phantomsection\label{cnd:rem2} For every $f\in D(A)$ the derivative 
$$t\mapsto \frac{d^+ T^t f}{dt}:=\wlim_{h\to 0^+} \frac{T^{t+h} f - T^t f}{h}$$
exists, is $*$-weak continuous from the right, and
{$$\frac{d^+ T^t f}{dt}=A T^t f =T^t A f\;\;\;\mbox{and}\;\;\;
T^t f-f=\int_0^t T^s Af\,ds\;\;\;\mbox{for all}\;\;\; t\geq 0.$$}
\item[(iii)]\phantomsection\label{cnd:rem3} For any $\beta>0$, the operator $\beta id-A:D(A)\to L_0(T)$ is invertible and the inverse operator $R_{\beta}:=(\beta id-A)^{-1}:L_0(T)\to D(A)$ (called the \emph{resolvent} of~$A$) is given by
$$R_{\beta} f(x)= \int_0^{\infty} e^{-\beta t}T^t f(x)\,dt,\;\;\;x\in X,\; f\in L_0(t).$$
\end{itemize}
\end{remark}

We are now in a position to prove the first main result of Section \ref{sec:cont}.
\begin{proof}[Proof of Theorem \ref{tw:rel}]
According to Lemma \ref{lem:lip}, we may consider the contraction semigroup \hbox{$\tP^t:C_b(X)\to C_b(X)$}, $t\geq 0$, given by 
 $\tP^t f= \bar{P}_{\varepsilon}^t f$ for $f\in C_b(X).$
From Lemma \ref{lem:st_cont} we know that $L_0(P)=C_b(X)$, and thus we can define the weak infinitesimal generator \hbox{$B:D(B)\to C_b(X)$} of the semigroup $(\tP^t)_{t\geq 0}$. 

Define the maps $Q^t:C_b(X)\to C_b(X)$, $t\geq 0$, by
$$Q^t f(y,i)=f(S_i(t,y),i)\;\;\;\mbox{for}\;\;\;(y,i)\in X.$$
Obviously, $(Q^t)_{t\geq 0}$ forms a contraction semigroup of linear operators and $L_0(Q)=C_b(X)$. Let $A: D(A)\to C_b(X)$ denote the weak infinitesimal generator of this semigroup.

We shall prove that $D(A)=D(B)$ and
\begin{equation}\label{niezc:e1}  Bf=\lambda W f+Af-\lambda f \,\;\;\;\mbox{for}\;\;\;f\in D(B), \end{equation}
where $W$ is determined by \eqref{def_W}. To do this, fix $f\in D(A)$, and let $\Psi_1$ be the function determined by \eqref{def:psi} for $n=1$. It now follows from Lemma \ref{lem:pt_apr} that there exists $u_f:X\times (0,\infty)\to\mathbb{R}$ such that
\begin{align*}
\frac{1}{t}\left(\widetilde{P}^t f(y,i) - f(y,i)\right)&=\lambda e^{-\lambda t}\frac{1}{t} \int_0^t  \Psi_1 f(y,i,s,t)\,ds+ e^{-\lambda t}\frac{1}{t}( f(S_i(t,y),i) -f(y,i))\\
&\quad-\lambda\frac{1- e^{-\lambda t}}{\lambda t}f(y,i)+\frac{u_f(y,i,t)}{t}\;\;\;\mbox{for}\;\;\;(y,i)\in X,\;t>0,
\end{align*}
and $\wlim_{t\to 0} u_f(\cdot,t)/t=0$. From the continuity of $f$, $S_j(\cdot,y)$, $w_{\theta}$ and the boundedness of $f$, we infer that the maps $(s,t)\mapsto \Psi_1 f(y,i,s,t)$ are continuous. Therefore, and by the mean value theorem, for each $t>0$ and $(y,i)\in X$, we can choose $s_{(y,i)}(t)\in(0,t)$ such that
$$\frac{1}{t}\int_0^t \Psi_1 f(y,i,s,t)\, ds = \Psi_1 f(y,i,s_{(y,i)}(t),t)\;\;\;\mbox{for}\;\;\;(y,i)\in X,\; t>0.$$
Keeping in mind the fact that $|\Psi_1|\leq \norma{f}_{\infty}$, we see that
$$ \wlim_{t\to 0} e^{-\lambda t} \frac{1}{t}\int_0^t \Psi_1 f(\cdot\,,s,t)\, ds = \Psi_1 f(\cdot\,,0,0)=Wf,$$
since the expression under the limit sign is bounded for all $(y,i)\in X$ and all $t$ in~some neighborhood of zero. Moreover, by the definition of $(Q^t)_{t\geq 0}$, we have
$$\frac{1}{t}(f(S_i(t,y),i) -f(y,i))=\frac{1}{t}(Q^t f(y,i) - f(y,i))\;\;\;\mbox{for} \;\;\;(y,i)\in X,\;t>0,$$
and $\wlim_{t\to 0} e^{-\lambda t}(Q^t f-f)/t=Af$ since $f\in D(A)$. Summarizing, we obtain
$$
\wlim_{t\to 0}\frac{1}{t}(\widetilde{P}^t f- f)=\lambda Wf+Af-\lambda f\in C_b(X),
$$
which gives \eqref{niezc:e1} and $D(A)\subset D(B)$. Clearly, letting $f\in D(B)$, we conclude analogously that $D(B)\subset D(A)$.

Let us now observe that $P=GW$, where $G$ and $W$ are defined by \eqref{def_G} and \eqref{def_W}, respectively. Indeed, we have for all
$f\in B_b(X)$ and $(y,i)\in X$,
\begin{align*}
GWf(y,i)&= \int_0^{\infty}\lambda e^{-\lambda t}Wf(S_i(t,y),i)\,dt\\ 
&=\sum_{j\in I}\int_0^{\infty} \int_{B_H(0,\varepsilon)}  \int_{\Theta} \lambda  e^{-\lambda t} f(w_{\theta}(S_i(t,y))+h,j)\pi_{ij}(w_{\theta}(S_i(t,y))+h)\\
&\hspace{3.4cm} \times p(S_i(t,y),\theta)\,d\theta \nu^{\varepsilon}(dh)dt=P f(y,i).
\end{align*}
Further, consider the resolvent $R_{\lambda}:C_b(X)\to D(A)$ of the operator $A$. As we pointed out in Remark \ref{rem:inf}\hyperref[cnd:rem3]{(iii)}, we have
$$R_{\lambda} f(y,i)=\int_0^{\infty} e^{-\lambda t} Q^t f(y,i)\,dt\;\;\;\mbox{for}\;\;\;f\in C_b(X),\,(y,i)\in X,$$
which implies $G|_{C_b(X)}=\lambda R_{\lambda}$. Hence, in particular,  $G(C_b(X))\subset D(A)$ and
\begin{equation}\label{resol_prop}
G(\lambda \operatorname{id} -A)f=(\lambda \operatorname{id} -A)G f=\lambda f\;\;\;\mbox{for}\;\;\; f\in D(A).
\end{equation}

We now proceed to show assertion \hyperref[cnd:1r]{(1)}. For this purpose, suppose that $P_{\varepsilon}$ has an invariant measure $\mu^*\in\mathcal{M}_1(X)$, and let $\nu^*:=\mu^* G$. Using the identity $P=GW$, we have 
$$\nu^* W= \mu^* GW=\mu^* P=\mu^*,$$
\begin{equation} \label{niezc:e4} 
\<f,\nu^*\>=\<G f, \mu^*\>=\< PG f, \mu^*\>=\<GWGf, \mu^*\>=\<WGf,\nu^*\>,\;\;\; f\in B_b(X).
\end{equation}
Letting $f\in D(A)=D(B)$, we can apply \eqref{niezc:e4} with $(\lambda \operatorname{id}-A)f$ in place of $f$, and use identity \eqref{resol_prop} to obtain
$$
\<(\lambda \operatorname{id}-A)f,\nu^*\>=\<WG(\lambda \operatorname{id}-A)f,\nu^*\>=\<\lambda W f,\nu^*\>.
$$
Consequently, from \eqref{niezc:e1} we then see that
\begin{equation}\label{niezc:e5} 
\<Bf,\nu^*\>=\< \lambda W f+Af-\lambda f,\nu^*\>=0\;\;\;\mbox{for}\;\;\;f \in D(B).
\end{equation}

From Remark \ref{rem:inf}\hyperref[cnd:rem2]{(ii)} we know that, for $f\in D(B)$, the map $s\mapsto \widetilde{P}^s Bf$ is $*$-weak continuous from the right and $\int_0^t \widetilde{P}^s Bf\,ds=\widetilde{P}^t f - f.$
Therefore, by \eqref{niezc:e5}, we obtain for $f\in D(B)$ and $t\geq 0$,
\begin{align*}
\<\widetilde{P}^t f -f,\nu^*\>=\<\int_0^t B\widetilde{P}^s f\,ds,\nu^*\>=\int_0^t\<B\widetilde{P}^s f,\nu^*\>\,ds=0,
\end{align*}
and thus $\< f,\nu^*\widetilde{P}^t\>=\<f,\nu^*\>.$ Since Remark \ref{rem:inf}\hyperref[cnd:rem1]{(i)} provides $C_b(X)\subset\wcl D(B)$, by using \eqref{wlimit}, we can conclude that this equality holds for all $f\in C_b(X)$. It then follows that $\nu^*$ is invariant for the semigroup $(\bar{P}_{\varepsilon}^t)_{t\geq 0}$, which completes the proof of \hyperref[cnd:1r]{(1)}. 

For the proof of statement \hyperref[cnd:2r]{(2)}, let $\nu^*\in\mathcal{M}_1(X)$ be an invariant measure for the semigroup $(\bar{P}_{\varepsilon}^t)_{t\geq 0}$, and set $\mu^*:=\nu^* W$. Letting $h\in D(B)$ and differentiating at $t=0$ the identity $\<\widetilde{P}^t h, \nu^*\>=\<h,\nu^*\>$, we obtain $\<Bh, \nu^*\>=0$. According to \eqref{niezc:e1} and the fact that $D(B)=D(A)$, this implies
$\<\lambda Wh,\nu^*\>=\<(\lambda \operatorname{id} - A)h, \nu^*\>$. If we now let $f\in D(A)$ and take $h=Gf$ then, in view of \eqref{resol_prop}, we infer that
$$\<f,\nu^*\>=\frac{1}{\lambda}\<({\lambda}\operatorname{id} - A)Gf, \nu^*\>=\<WGf, \nu^*\>.$$
Since $C_b(X)\subset\wcl D(A)$, it is clear that the latter equality holds for any \hbox{$f\in C_b(X)$}, and therefore $\nu^*=\nu^*WG$. Finally, using the fact that $P=GW$, we obtain \hbox{$\mu^* G=\nu^*WG=\nu^*$}, as well as
$$\mu^*=\nu^* W=(\nu^* WG)W=(\nu^*W)(GW)=\mu^* P,$$
which completes the proof.
\end{proof}

\subsubsection{Proof of Theorem \ref{SLLN_cont}} \label{sec:proof4}
It now remains to prove Theorem \ref{SLLN_cont}, which provides the SLLN for the examined PDMP. The main idea of our approach is based on comparision of the averages  $t^{-1} \int_0^t f({\overline{Y}(s),\overline{\xi}(s)})\,ds$ and $n^{-1} \sum_{k=0}^{N_t-1} Gf({Y_k,\xi_k})$, where $f\in Lip_b(X)$. We {aim to }show that the difference between them vanishes as $t\to\infty$, which enables the application of Theorem \ref{tw:slln-dyskr}. The rest then follows from Theorem \ref{tw:rel}. In order to show that the above-mentioned limit holds, we use arguments similar to those in the proof of \cite[Lemma 2.5]{b:benaim1}.

\begin{lemma} \label{lem:cd_conv}
For every $f\in B_b(X)$ we have
$$\lim_{t\to \infty} \left(\frac{1}{t} \int_0^t f\left({\overline{Y}(s),\overline{\xi}(s)}\right)\,ds-\frac{1}{N_t}\sum_{k=0}^{N_t-1} Gf(Y_k,\xi_k) \right)=0 \;\;\;\mbox{{a.s.}},$$
where $G:B_b(X)\to B_b(X)$ is determined by \eqref{def_G}.
\end{lemma}
\begin{proof}
Let $f\in B_b(X)$, and define
$$
M_0:=0,\;\;\;M_n:=\sum_{k=0}^{n-1} \left(\int_{\tau_k}^{\tau_{k+1}} f\left(\overline{Y}(s),\overline{\xi}(s)\right)\,ds-\frac{1}{\lambda} Gf(Y_k,\xi_k)\right)\;\;\;\mbox{for}\;\;\;n\in\n.
$$
Moreover, put $\Delta \tau_0:=0$. We will first show that $(M_n)_{n\in\n_0}$ is a martingale with respect to the natural filtration of $(Y_n, \xi_n,\Delta \tau_n)_{n\in\n_0}$, further denoted $(\mathcal{F}_n)_{n\in\n_0}$, and that the increments of $(M_n)_{n\in\n_0}$ are uniformly bounded in the $\mathcal{L}^2(\mathbb{P})$-norm. For this purpose, let us define $F:X\times\mathbb{R}_+\to \mathbb{R}$ by
$$F(y,i,t)=\int_0^t f(S_i(s,y),i)\,ds\;\;\;\mbox{for}\;\;\;(y,i)\in X,\;t\geq 0.$$
We then see that $F$ is Borel measurable, and due to \eqref{zz}, we can write
$$\int_{\tau_k}^{\tau_{k+1}} f\left({\overline{Y}(s),\overline{\xi}(s)}\right)\,ds=\int_{\tau_k}^{\tau_{k+1}} f(S_{\xi_k}(s-\tau_k,Y_k),\xi_k)\,ds=F(Y_k,\xi_k,\Delta \tau_{k+1})
$$
for every $k\in\n_0$. This shows that $(M_n)_{n\in\n}$ is adapted to $(\mathcal{F}_n)_{n\in\n}$  and gives
\begin{equation} \label{mar:int} 
M_{n+1}-M_n=F(Y_n,\xi_n,\Delta \tau_{n+1})-\lambda^{-1}Gf(Y_n,\xi_n)\;\;\;\mbox{for}\;\;\;n\in\n_0.
\end{equation}
We now observe that, for any $n\in\n_0$ and any $(y,i,u)\in X\times \mathbb{R}_+$,
\begin{align*}
&\ew\lbrack F(Y_n,\xi_n,\Delta \tau_{n+1})\,|\,Y_n=y,\,\xi_n=i,\,\Delta\tau_n=u\rbrack=\int_0^{\infty} F(y,i,t)\,\pr(\Delta \tau_{n+1}\in dt)\\
&=\int_0^{\infty} \lambda e^{-\lambda t}\left(\int_0^t  f(S_i(s,y),i)\,ds\right)\,dt=\int_0^{\infty}\left(\int_s^{\infty} \lambda e^{-\lambda t}\,dt \right)f(S_i(s,y),i)ds\\
&=\int_0^{\infty} e^{-\lambda s} f(S_i(s,y),i)\,ds=\lambda^{-1} G f(y,i).
\end{align*}
Consequently, using the Markov property we can conclude that
\begin{align} \label{mar:war}
\ew\lbrack F(Y_n,\xi_n,\Delta \tau_{n+1})\,|\,\mathcal{F}_n\rbrack&=\ew\lbrack F(Y_n,\xi_n,\Delta \tau_{n+1})\,|\,Y_n,\,\xi_n,\,\Delta\tau_n\rbrack \nonumber\\
&=\lambda^{-1} G f(Y_n,\xi_n)\;\;\;\mbox{for all}\;\;\;n\in\n_0.
\end{align}
From \eqref{mar:int} and \eqref{mar:war} it now follows that
\begin{align*}
\ew\lbrack M_{n+1}-M_n\,|\,\mathcal{F}_n\rbrack &= \ew\left[F(Y_n,\xi_n,\Delta \tau_{n+1})\,|\mathcal{F}_n\right]-\lambda^{-1} Gf(Y_n,\xi_n)=0,
\end{align*}
whence $(M_n)_{n\in\n}$ is a martingale.
Finally, using the fact that $|F(\cdot,t)|\leq t\norma{f}_{\infty}$ for all $t\geq 0$, and $|Gf|\leq \norma{f}_{\infty}$, we obtain for all $n\in\n_0$,
\begin{align*}
\ew_x\left[(M_{n+1}-M_n)^2\right]&=\ew_x\left[\, \left(F(Y_n,\xi_n,\Delta \tau_{n+1})-\lambda^{-1}  Gf(Y_n,\xi_n)\right)^2 \,\right]\\
& \leq 2 \ew_x \lbrack F(Y_n,\xi_n,\Delta \tau_{n+1})^2 \rbrack + 2 \ew_x \left[\lambda^{-2} Gf(Y_n,\xi_n)^2 \right] \\
&\leq 2\norma{f}_{\infty}^2\left(\ew\lbrack (\Delta\tau_{n+1})^2\rbrack + \frac{1}{\lambda^2}\right)=6\lambda^{-1}\norma{f}_{\infty}^2.
\end{align*}
The claim stated at the beginning of the proof is now established. 

Let us now define
$$r_t=\frac{1}{N_t}\int_{\tau_{N_t}}^t f\left({\overline{Y}(s),\overline{\xi}(s)}\right)\,ds\;\;\;\mbox{if}\;\;\;\tau_1\leq t,\;\;\;\mbox{and}\;\;\;r_t=0\;\;\;\mbox{otherwise}.$$
Then
\begin{equation}\label{rest_esti}|r_t|\leq \norma{f}_{\infty}\frac{\Delta \tau_{N_t+1}}{N_t}\;\;\;\mbox{whenever}\;\;\;t\geq \tau_1,\end{equation}
and, for $\tau_1\leq t$, we can write
\begin{equation} 
\frac{1}{t}\int_0^t f\left({\overline{Y}(s),\overline{\xi}(s)}\right)\,ds=\frac{N_t}{t}\frac{1}{N_t}\sum_{k=0}^{N_t-1}\int_{\tau_k}^{\tau_{k+1}} f\left({\overline{Y}(s),\overline{\xi}(s)}\right)\,ds+\frac{N_t}{t}r_t .
\end{equation}
Consequently, keeping in mind the definition of $M_n$, we obtain
\begin{align*}
&\frac{1}{t}\int_0^t f\left({\overline{Y}(s),\overline{\xi}(s)}\right)\,ds-\frac{1}{N_t}\sum_{k=0}^{N_t-1} Gf\left({Y_k,\xi_k}\right)\\
&=\frac{N_t}{t}\frac{M_{N_t}}{N_t}+\frac{1}{\lambda}\left(\frac{N_t}{t}-\lambda\right)\frac{1}{N_t}\sum_{k=0}^{N_t-1} Gf({Y_k,\xi_k})+\frac{N_t}{t}r_t\;\;\;\mbox{for}\;\;\;t\geq \tau_1.
\end{align*}
We now only need to show that the right-hand side of the above equality tends a.s to~$0$, as $t\to \infty$. To do this, we first note that $n^{-1}\sum_{k=0}^{n-1} Gf(Y_k,\xi_k)$ is bounded (by $\norma{f}_{\infty}$) for all $n$.
Further, from the Elementary Renewal Theorem, we know that $N_t/t\to \lambda$ a.s, and thus $N_t\to \infty$ a.s. It then follows from the Borel--Cantelli Lemma that $\Delta\tau_{N_t+1}/N_t \to 0$ a.s, since 
$$\sum_{n=1}^{\infty} \pr\left(\Delta \tau_{n+1}/n\geq  \varepsilon\right)=\sum_{n=1}^{\infty} e^{-\lambda n \varepsilon}<\infty\;\;\;\mbox{for any}\;\;\; \varepsilon>0.$$ This together with \eqref{rest_esti} ensures that $r_t\to 0$ (a.s). 
Finally, by the SLLN for martingales \hbox{(\cite[Theorem 2.18]{b:limit_mar})} we have $M_k/k\to 0$ a.s., as $k\to \infty$, and thus $M_{N_t}/{N_t}\to 0$ a.s. This yields the desired conclusion and completes the~proof.
\end{proof}
Before proceeding to the proof of Theorem \ref{SLLN_cont}, let us recall that, apart from the requirement that conditions \textnormal{\hyperref[cnd:A1]{(A1)}-\hyperref[cnd:A5]{(A5)}} hold with parameters satisfying \eqref{ll}, we have made an additional assumption that $\mathcal{L}$, appearing in \hyperref[cnd:A2]{(A2)}, is constant. Hence, for some $\bar{L}$, we can write $\mathcal{L}(t)=\bar{L}$ for all $t\geq 0$.

\begin{proof}[Proof of Theorem \ref{SLLN_cont}]
Fix $f\in Lip_b(X)$, ${(y,i)}\in X$, and let $\mu^*$ denote the unique invariant distribution of $(Y_n,\xi_n)_{n\in\n_0}$ (which exists by Theorem \ref{tw:tempo}). It follows from Theorem~\ref{tw:rel} that $\nu^*=\mu^*G$. Therefore, in view of Theorem \ref{tw:slln-dyskr}, it suffices to show that $Gf\in Lip_b(X)$, since then
$$\lim_{t\to \infty} \frac{1}{N_t} \sum_{k=0}^{N_t-1}Gf({Y_k,\xi_k})=\<Gf,\mu^*\>=\<f,\nu^*\>\;\;\;\pr_{(y,i)} - a.s.,$$
which by Lemma \ref{lem:cd_conv} gives the desired claim. 

For this purpose, let $(y,i),(z,l)\in X$. By condition \hyperref[cnd:A2]{(A2)} we then have
\begin{align*}
|Gf(y,i)-Gf(z,l)|&\leq \int_0^{\infty} \lambda e^{-\lambda t}|f(S_{i}(t,y),i)-f(S_{l}(t,z),l)|\,ds\\
&\leq \lambda |f|_{Lip}  \int_0^{\infty} e^{-\lambda t}(\norma{S_i(t,y)-S_l(t,z)}+c\disc(i,l))\,dt\\
&\leq\lambda |f|_{Lip} \left[\frac{L}{\lambda-\alpha}\norma{y-z}+\left(\frac{\bar{L}}{\lambda^2}+\frac{c}{\lambda} \right)\disc(i,l)\right].\\
\end{align*}
Since $\bar{L}=M_{\mathcal{L}}\leq c L$ due to \eqref{nc}, we see that
$$
|Gf(y,i)-Gf(z,l)|\leq \lambda |f|_{Lip} \left(\frac{L}{\lambda-\alpha}+\frac{L}{\lambda^2}+\frac{1}{\lambda}\right)\rho_c((y,i),(z,l)).
$$
The proof is now complete.
\end{proof}

\section*{Appendix}\label{appendix}
To illustrate the main idea behind Theorem \ref{ks-stab} and simultaneously assure a certain level of self-containedness, we shall give a general sketch of its proof. Details can be found in \cite{b:kapica}. As we have noted in Section \ref{section2}, for a given substochastic kernel $Q:E^2\times \mathcal{B}(E^2)\to\left [0,1\right]$ satisfying \eqref{qeq}, we can define a coupling of $P$ whose transition function takes the form $B=Q+R$.  In order to distinguish the case where the next step of the coulpling is drawn only according to $Q$ from the case when it is determined only by $R$, one can consider the augmented space $\widehat{E^2}:=E^2\times\{0,1\}=(E^2 \times \{0\})\cup(E^2 \times \{1\}) $ and the stochastic kernel $B$ on 
$\widehat{E^2}\times \mathcal{B}(\widehat{E^2})$ given by 
$$\widehat{B}(x,y,\theta,C):=(\delta_0\otimes R(x,y,\cdot))(C)+(\delta_1 \otimes Q(x,y,\cdot))(C)$$
for $(x,y,\theta)\in\widehat{E^2}$ and $C\in \mathcal{B}(\widehat{E^2})$, where $\delta_0$ and $\delta_1$ stand for the Dirac measures. Let $\widehat{\Phi}_n:=(\Phi_n^1, \Phi_n^2, \theta_n)$, $n\in\mathbb{N}_0$, be an \hbox{$\widehat{E^2}$-valued} Markov chain with transition function $\widehat{B}$, defined on a suitable probability space $(\widehat{\Omega}, \mathcal{\widehat{A}},\widehat{\pr})$ (for instance, one can take $\widehat{\Omega}=(\widehat{E^2})^{\mathbb{N}_0}$ and $\mathcal{\widehat{A}}=\mathcal{B}(\widehat{\Omega}))$. Then, for any $(x,y,\theta)\in\widehat{E}^2$ and $A\in\mathcal{B}(E^2)$,
$$\widehat{\pr}_{x,y,\theta}(\widehat{\Phi}_n\in A\times \{1\})=Q^n(x,y, A),\;\;\;\widehat{\pr}_{x,y,\theta}(\widehat{\Phi}_n\in A\times \{0\})=R^n(x,y, A),$$
$$\mbox{and}\;\;\;\widehat{\pr}_{x,y,\theta}((\Phi_n^1,\Phi_n^1)\in A)=B^n(x,y,A)$$
(independently of $\theta\in\{0,1\}$), where $Q^n(\cdot,A):=Q^n \mathbbm{1}_A$, and the analogous notation is employed for the kernels $R$ and $B$. Define $\tau$ to be an absorption time on $E^2\times\{1\}$, that is,
$$\tau:=\inf\{n\in\n:\, \widehat{\Phi}_m\in E\times\{1\}\;\;\mbox{for}\;\;m\geq n\}=\inf\{n\in\n:\, \theta_m=1\;\;\mbox{for}\;\;m\geq n\}.$$
Further, for $A\in \mathcal{B}(E^2)$ and $k\in \n$, let $\kappa^{(k)}_A$ denote the first return time to the set $A$ after time $k-1$, i.e. $\kappa^{(k)}_A:=\inf\{n\in\n: n\geq k,\;(\Phi_n^1,\Phi_n^2)\in A\}$. 

Then, for any $f\in\mathcal{R}_{FM}(E)$, $x,y\in E$ and  $n,M,N\in\n$ satisfying $n>M>N$, we have
\begin{align}
\label{com0}
\begin{split}
|P^n f(x) - P^n f(y)|&\leq \int_{E^2} |f(u)-f(v)|\, B^n(x,y,du,dv)\\
&\leq \int_{E^2} \rho(u,v)\,\widehat{\pr}_{x,y,\theta}\left(\kappa^{(N)}_K \leq M,\,\tau\leq N,\,\left(\Phi_n^1,\Phi_n^2\right)\in du\times dv\right)\\
&\quad+2\widehat{\pr}_{x,y,\theta}\left(\kappa^{(N)}_K>M\right)+2\widehat{\pr}_{x,y,\theta}(\tau>N),
\end{split}
\end{align}
where $K$ is given by \eqref{zb_k}. Let $I_j$ (where $j\in\{1,2,3\}$) denote the $j$th component on the right-hand side of the above inequality. The contraction property assumed for $Q$ in \hyperref[cnd:B3]{(B3)} allows one to show that \hbox{$I_1\leq \Gamma q^{n-M}$} for some $\Gamma>0$ (where $q\in(0,1)$). Condition \hyperref[cnd:B1]{(B1)} guarantees that $\overline{V}(x,y):=V(x)+V(y)$ is a Lyapunov function on $E^2$ and $B\overline{V}\leq a \overline{V}+2b$. It then follows (from \cite[Lemma 2.2]{b:kapica}) that for \hbox{$J:=\{(u,v):\, \overline{V}(u,v)<4b(1-a)^{-1}\}$} there exist constants $\delta_{\kappa}\in (0,1)$ and $C_{\kappa}$ such that $\widehat{\ew}_{x,y,\theta}\Big(\delta_{\kappa}^{-\kappa^{(N)}_J}\Big)\leq \delta_{\kappa}^{-N}C_{\kappa}(1+\overline{V}(x,y))$. This together with the strong Markov property and condition \eqref{fastint}, assumed in \hyperref[cnd:B2]{(B2)}, gives  
\begin{equation}\label{com:1}\widehat{\ew}_{x,y,\theta}\Big(\delta_{\kappa}^{-\kappa^{(N)}_K}\Big)\leq \delta_{\kappa}^{-N}\bar{C}C_{\kappa}(1+\overline{V}(x,y)).\end{equation} 
Hence
$I_2\leq 2\bar{C}C_{\kappa}\,\widehat{\delta}^{M-dN}(1+\overline{V}(x,y))$ with $\widehat{\delta}:=\max\{\delta_{\kappa},\zeta\}$ and some $d\geq 1$, where $\bar{C}$ and $\zeta$ are the constants occuring in \eqref{fastint}. To estimate $I_3$ we first observe that, for some $\eta>0$, and for all $k\in\n$ and $(u,v)\in F$,
$$Q^k(u,v,U(q^k\rho(u,v)))\geq \eta^k\;\;\;\mbox{and}\;\;\;Q^k(u,v,X^2)\geq 1-\frac{l}{1-q^{\nu}}\rho^{\nu}(u,v).$$
The first inequality can be obtained inductively from condition \hyperref[cnd:B4]{(B4)}, and the second one follows from \hyperref[cnd:B3]{(B3)} and \hyperref[cnd:B5]{(B5)}. These estimates further imply that
\begin{equation}\label{com:2}\widehat{\pr}_{u,v,\theta}\left(\widehat{\Phi}_k\in E^2\times\{1\}\;\;\mbox{for all}\;\;k\in\n\right)=\lim_{k\to\infty} Q^k(u,v,E^2)\geq p,\;\;\;(u,v)\in F,\end{equation}
for some $p>0$. Let $\epsilon:=\kappa^{(1)}_{E^2\times\{0\}}$. Then hypotheses \hyperref[cnd:B3]{(B3)} and \hyperref[cnd:B5]{(B5)} also provide the existence of $\delta_{\epsilon}\in (0,1)$ and \hbox{$C_{\epsilon}>0$} such that  
\begin{equation}\label{com:3}
\widehat{\ew}_{u,v,\theta} \left(\mathbbm{1}_{\{\epsilon<\infty\}}\delta_{\epsilon}^{-\epsilon}\right)\leq C_{\epsilon}\;\;\;\mbox{for}\;\;\; (u,v)\in F.\end{equation}
Having established \eqref{com:1} - \eqref{com:3} (and keeping in mind that $F\subset K$), one can show (as in lemma \cite[Lemma 2.2]{b:kapica}) that there exist $\delta_\tau\in (0,1)$ and $C_{\tau}>0$ for which
$\widehat{\ew}_{x,y,\theta}(\delta_{\tau}^{-\tau})\leq C_{\tau}(1+\overline{V}(x,y))$, whence $I_3\leq 2C_{\tau}\delta_{\tau}^N(1+\overline{V}(x,y))$. It now follows easily from \eqref{com0} that there exist $\beta\in (0,1)$ and $C\in\mathbb{R}$ such that
$$|P^n f(x) - P^{n} f(y)|\leq C\beta^n(V(x)+V(y)+1) \;\;\;\mbox{for}\;\;\;f\in\mathcal{R}_{FM}(E),\,x,y\in E,\,n\in\n.$$
Consequently, for any $\mu,\nu\in\mathcal{M}_1(E)$ satisfying $\<V,\mu\><\infty,\,\<V,\nu\><\infty$, we obtain
\begin{equation}\label{com:4} d_{FM}(\mu P^n,\nu P^n)\leq C\beta^n(\<V,\mu\>+\<V,\nu\>+1)\;\;\;\mbox{for}\;\;\;n\in\n.\end{equation}
Let $x_0\in E$ and put $\mu_0:=\delta_{x_0}$ (the Dirack measure at $x_0$). Applying  \eqref{com:4} with \hbox{$\mu=\mu_0$} and \hbox{$\nu=\mu_0 P^k$} for each $k\in\mathbb{N}_0$ and condition \hyperref[cnd:B5]{(B5)}, we see that $(\mu_0 P^k)_{k\in\n_0}$ is a Cauchy sequence with respect to $d_{FM}$. Since $(\mathcal{M}_1(E),d_{FM})$ is complete, \hbox{$\mu_0 P^k \stackrel{w}{\to }\mu^*$} (as $k\to \infty$) for some \hbox{$\mu^*\in\mathcal{M}_1(E)$}. The assumption of the Feller property ensures that $\mu^*$ is invariant  \hbox{for $P$}. Clearly, due to \hyperref[cnd:B1]{(B1)}, $V$ is integrable with respect to each invariant measure, whence \hbox{$\<V, \mu^*\><\infty$} and $\mu^*$ is a unique invariant probability measure of $P$.

\section*{Acknowledgement}
We acknowledge useful discussions with S.C. Hille. H. Wojew\'odka is supported by the Foundation for Polish Science (FNP).

\footnotesize
\bibliography{references}

\begin{thebibliography}{10}

\bibitem{b:biotech1}
B.~Alberts, D.~Bray, K.~Hopkin, A.~Johnson, J.~Lewis, M.~Raff, K.~Roberts, and
  P.~Walter.
\newblock {\em Essential Cell Biology, 4th edition}.
\newblock Garland Science, New York, 2013.

\bibitem{b:benaim3}
M.~Bena{\"i}m, S.~Le~Borgne, F.~Malrieu, and P.-A. Zitt.
\newblock Quantitative ergodicity for some switched dynamical systems.
\newblock {\em Electron. Commun. Probab.}, 17(56, 14), 2012.

\bibitem{b:benaim1}
M.~Bena{\"i}m, S.~Le~Borgne, F.~Malrieu, and P.-A. Zitt.
\newblock Qualitative properties of certain piecewise deterministic
  \text{Markov} processes.
\newblock {\em Ann. Inst. Henri Poincar Probab.}, 51(3):1040--1075, 2014.

\bibitem{b:costa}
O.L.V. Costa and F.~Dufour.
\newblock Stability and ergodicity of piecewise deterministic \text{Markov}
  processes.
\newblock {\em SIAM J.~Control Optim.}, 47(2):1053--1077, 2008.

\bibitem{b:crandall}
M.~Crandall and T.~Ligget.
\newblock Generation of semigroups of nonlinear transformations on general
  \text{B}anach spaces.
\newblock {\em Amer. J. Math.}, 93:265--298, 1971.

\bibitem{b:dudley}
R.M. Dudley.
\newblock Convergence of \text{B}aire measures.
\newblock {\em Studia Math.}, 27:251--268, 1966.

\bibitem{b:dynkinb}
E.B. Dynkin.
\newblock {\em Markov processes I}.
\newblock Springer-Verlag, Berlin, 1965.

\bibitem{b:dynkin_mar}
E.B. Dynkin.
\newblock {\em Selected papers of E.B. Dynkin with commentary}.
\newblock American Mathematical Society, RI; International Press, Cambridge,
  MA, 2000.

\bibitem{b:biotech2}
M.~Fitzgerald-Hayes and F.~Reichsman.
\newblock {\em DNA and Biotechnology, 3rd edition}.
\newblock Academic Press/Elsevier Inc., Burlington, 2010.

\bibitem{b:hairer}
M.~Hairer.
\newblock Exponential mixing properties of stochastic \text{PDEs} through
  asymptotic coupling.
\newblock {\em Probab. Theory Related Fields}, 124(3):345--380, 2002.

\bibitem{b:limit_mar}
P.~Hall and C.C. Heyde.
\newblock {\em Martingale limits theory and its applications}.
\newblock Academic Press, New York, 1980.

\bibitem{b:HHS}
S.~Hille, K.~Horbacz, and T.~Szarek.
\newblock Existence of a unique invariant measure for a class of equicontinuous
  \text{M}arkov operators with application to a stochastic model for an
  autoregulated gene.
\newblock {\em Ann. Math. Blaise Pascal}, 23(2):171--217, 2016.

\bibitem{b:CLT}
S.C. Hille, K.~Horbacz, T.~Szarek, and H.~Wojew\'{o}dka.
\newblock Limit theorems for some \text{Markov} chains.
\newblock {\em J.~Math. Anal. Appl.}, 443(1):385--408, 2016.

\bibitem{b:horbacz_diss}
K.~Horbacz.
\newblock Invariant measures for random dynamical systems.
\newblock {\em Dissertationes Math.}, 451, 2008.

\bibitem{b:kazufami}
K.~Ito and F.~Kappel.
\newblock {\em Evolution equations and approximations}.
\newblock Ser. Adv. Math. Appl. Sci. 61, World Scientific, New Jersey, 2002.

\bibitem{b:kapica}
R.~Kapica and M.~\'Sl\k{e}czka.
\newblock Random iterations with place dependent probabilities.
\newblock {\em arXiv}, 1107.0707v2, 2012 (submitted to J. Math. Anal. Appl.).

\bibitem{b:las_frac}
A.~Lasota.
\newblock From fractals to stochastic differential equations, in:
  Chaos--\text{T}he \text{I}nterplay \text{B}etween \text{S}tochastic and
  \text{D}eterministic behaviour (\text{P}roceedings of the \text{XXXIst}
  \text{W}inter \text{S}chool of \text{Theoretical Physics}, \text{K}arpacz,
  \text{P}oland 13-24 \text{F}eb 1995), \text{E}ds. \text{P. G}arbaczewski,
  \text{M. W}olf and \text{A. W}eron.
\newblock {\em Lecture Notes in Physics}, 457:235--255, 1995.

\bibitem{b:cells}
A.~Lasota and M.C. Mackey.
\newblock Cell division and the stability of cellular populations.
\newblock {\em J.~Math. Biol.}, 38:241--261, 1999.

\bibitem{b:mackey_tyran}
M.C. Mackey, M.~Tyran-Kami\'{n}ska, and R.~Yvinec.
\newblock Dynamic behavior of stochastic gene expression models in the presence
  of bursting.
\newblock {\em SIAM J. Appl. Math.}, 73(5):1830--1852, 2013.

\bibitem{b:meyn}
S.P. Meyn and R.L. Tweedie.
\newblock {\em Markov chains and stochastic stability}.
\newblock Springer-Verlag, London, 1993.

\bibitem{b:paz}
A.~Pazy.
\newblock {\em Semigroups of Linear Operators and Applications to Partial
  Differential Equations}.
\newblock Springer-Verlag, New York, 1999.

\bibitem{b:revuz}
D.~Revuz.
\newblock {\em Markov chains}.
\newblock North-Holland Elsevier, Amsterdam, 1975.

\bibitem{b:shir}
A.~Shirikyan.
\newblock A version of the law of large numbers and applications.
\newblock {\em In: Probabilistic Methods in Fluids, Proceedings of the Swansea
  Workshop held on 14 - 19 April 2002, World Scientific, New Jersey}, pages
  263--271, 2003.

\bibitem{b:sleczka}
M.~\'Sl\k{e}czka.
\newblock Exponential convergence for \text{Markov} systems.
\newblock {\em Ann. Math. Sil.}, 29:139--149, 2015.

\bibitem{b:woj}
H.~Wojew\'odka.
\newblock Exponential rate of convergence for some \text{Markov} operators.
\newblock {\em Statist. Probab. Lett.}, 83(10):2337--2347, 2013.

\end{thebibliography}
\bibliographystyle{plain}
\end{document}